\documentclass[12pt, reqno]{amsart}
\usepackage{amssymb,latexsym,amsmath,amscd,amsthm,graphicx, color}
\usepackage{enumitem}
\usepackage[ruled,vlined]{algorithm2e}

\usepackage[all]{xy}
\usepackage{pgf,tikz}
\usepackage{mathrsfs}
\usetikzlibrary{arrows}
\usepackage{pgfplots}
\pgfplotsset{compat=1.15}
\usepackage[left=0.6 in, top=0.6 in, right=0.6 in, bottom=0.3 in]{geometry}
\usepackage{changepage}
\allowdisplaybreaks
\makeatletter
\newcommand{\setword}[2]{%
	\phantomsection
	#1\def\@currentlabel{\unexpanded{#1}}\label{#2}%
}
\makeatother

\usepackage[linkcolor=blue, urlcolor=blue, citecolor=blue,
colorlinks, bookmarks]{hyperref}

\definecolor{uuuuuu}{rgb}{0.26666666666666666,0.26666666666666666,0.26666666666666666}
\definecolor{xdxdff}{rgb}{0.49019607843137253,0.49019607843137253,1.}
\definecolor{ffqqqq}{rgb}{1.,0.,0.}
\definecolor{ffqqqq}{rgb}{1.,0.,0.}
\definecolor{ffxfqq}{rgb}{1.,0.4980392156862745,0.}

\definecolor{uuuuuu}{rgb}{0.26666666666666666,0.26666666666666666,0.26666666666666666}
\definecolor{qqwuqq}{rgb}{0.,0.39215686274509803,0.}
\definecolor{zzttqq}{rgb}{0.6,0.2,0.}
\definecolor{xdxdff}{rgb}{0.49019607843137253,0.49019607843137253,1.}
\definecolor{qqqqff}{rgb}{0.,0.,1.}
\definecolor{cqcqcq}{rgb}{0.7529411764705882,0.7529411764705882,0.7529411764705882}
\definecolor{sqsqsq}{rgb}{0.12549019607843137,0.12549019607843137,0.12549019607843137}
\definecolor{uuuuuu}{rgb}{0.26666666666666666,0.26666666666666666,0.26666666666666666}
\definecolor{ffqqqq}{rgb}{1,0,0}
\definecolor{xdxdff}{rgb}{0.49019607843137253,0.49019607843137253,1}

\definecolor{yqqqyq}{rgb}{0.5019607843137255,0,0.5019607843137255}
\definecolor{qqqqff}{rgb}{0,0,1}
\definecolor{ffqqqq}{rgb}{1,0,0}
\definecolor{ffqqff}{rgb}{1,0,1}

\theoremstyle{plain}

\newtheorem{theorem}[subsection]{Theorem}

\newtheorem{corollary}[subsection]{Corollary}
\newtheorem{lemma}[subsection]{Lemma}
\newtheorem{defi}[subsection]{Definition}

\newtheorem{prop}[subsection]{Proposition}

\theoremstyle{definition}

\newtheorem{exam}[subsection]{Example}

\newtheorem{remark}[subsection]{Remark}

\newtheorem{note}[subsection]{Note}

\newcommand{\sci}{\subset}

\newcommand{\ga}{\alpha}

\newcommand{\tit}{\textit}

\newcommand{\D}[1]{\mathbb{#1}}

\newcommand{\te}{\text}

\newcommand{\nd}{\noindent}

\begin{document}
  \nd To appear,\tit{ Mathematics} 
\title{Optimal Quantization of Finite Uniform Data on the Sphere}

	\author{Mrinal Kanti Roychowdhury} 

\address{School of Mathematical and Statistical Sciences\\
		University of Texas Rio Grande Valley\\
		1201 West University Drive\\
		Edinburg, TX 78539-2999, USA.}
	
	\email{mrinal.roychowdhury@utrgv.edu}

\subjclass[2020] {Primary 62H30, 41A25; Secondary 60Exx, 94A34, 68U05.}
	\keywords {Optimal quantization; spherical quantization; centroidal Voronoi tessellations; discrete uniform distributions; geodesic distance; intrinsic (Karcher) mean; Lipschitz stability; ring–structured data; Lloyd’s algorithm on manifolds.}
	
	\date{}
	\maketitle
	
	\pagestyle{myheadings}\markboth{Mrinal Kanti Roychowdhury}{Optimal Quantization of Finite Uniform Data on the Sphere}

\begin{abstract}
This paper develops a systematic and geometric theory of optimal quantization on the unit sphere $\mathbb S^2$, focusing on finite uniform probability distributions supported on the spherical surface---rather than on lower-dimensional geodesic subsets such as circles or arcs. We first establish the existence of optimal sets of $n$-means and characterize them through centroidal spherical Voronoi tessellations. Three fundamental structural results are obtained. 
First, a \emph{cluster--purity theorem} shows that when the support consists of well-separated components, each optimal Voronoi region remains confined to a single component. Second, a \emph{ring--allocation (discrete water--filling) theorem} provides an explicit rule describing how optimal representatives are distributed across multiple latitudinal rings, together with closed-form distortion formulas. Third, a \emph{Lipschitz--type stability theorem} quantifies the robustness of optimal configurations under small geodesic perturbations of the support. 
In addition, a spherical analogue of Lloyd’s algorithm is presented, in which intrinsic (Karcher) means replace Euclidean centroids for iterative refinement. These results collectively provide a unified and transparent framework for understanding the geometric and algorithmic structure of optimal quantization on $\mathbb S^2$.
\end{abstract}

\section{Introduction}

 Optimal quantization for a probability distribution refers to the idea of estimating a given probability with support containing an infinite (discrete or continuous) or finite number of values  by a discrete probability with support containing a finite or smaller number of values.   
It has broad applications in communications, information theory, signal processing, and data compression (see \cite{GarshoGray,  GregoryLinder2002, GrayNeuhoff1998, Pollard1982, Zador1982, Zamir2014}). The monograph of Graf and Luschgy~\cite{GrafLuschgy2000} provides a systematic and rigorous treatment of quantization for probability distributions, including existence, uniqueness, and high--resolution asymptotics.

Let $(M, d)$ be a metric space. 
Let $P$ be a Borel probability measure on $M$  and $r \in (0, \infty)$. Let $\ga$ be a locally finite (i.e., intersection of $\ga$ with any bounded subset of $M$ is finite) subset of $M$. This implies that $\ga$ is countable and closed. Then, the distortion error for $P$, of order $r$, with respect to the set $\ga \sci M$, denoted by $V_r(P; \ga)$, is defined by
\begin{equation*}
	V_r(P; \ga)= \int \mathop{\min}\limits_{a\in\ga} d(x, a)^r dP(x).
\end{equation*}
Then, for $n\in \mathbb{N}$, the \tit {$n$th quantization
	error} for $P$, of order $r$, is defined by
\begin{equation} \label{Vr} V_{n, r}:=V_{n, r}(P)=\inf \Big\{V_r(P; \ga) : \ga \sci M, ~ 1\leq  \text{card}(\ga) \leq n \Big\},\end{equation}
where $\te{card}(A)$ represents the cardinality of a set $A$. For the probability measure $P$, we assume that 
\begin{equation*} \label{M1eq}
	\int d(x, 0)^r dP(x)<\infty. 
\end{equation*}
Then, there is a set $\ga$ for which the infimum in \eqref{Vr} exists (see \cite{GrafLuschgy2000}). 
A set $\ga$ for which the infimum in \eqref{Vr} exists and does not contain more than $n$ elements is called an optimal set of $n$-means for $P$ of order $r$. Typically, in the literature it has been taken $r=2$. If the support of $P$ contains infinitely many elements, then an optimal set of $n$-means always contains exactly $n$ elements (see \cite {GrafLuschgy2000}). When $M=R^k$, where $k\in\D N$ and $d$ is the standard Euclidean metric, then the quantization problems for Borel probability measures $P$ have been extensively studied. 
For some recent work in this direction, one can see \cite{DFG, DettmannRoychowdhury, GrafLuschgy1997, GrafLuschgy2004, KNZ, PRRSS, Potzelberger, R1, R2, R3, R4, RS}. 
 The Euclidean case benefits from linear structure, convexity, and the availability of explicit centroids, making it possible to derive fine properties of optimal $n$--means and associated Voronoi partitions.

When the underlying space $(M,d)$ is non--Euclidean, the quantization problem becomes substantially more delicate. On Riemannian manifolds, distances are measured intrinsically along geodesics, and curvature directly influences both the geometry of Voronoi regions and the structure of optimal representatives. In particular, the absence of global linear structure and convexity prevents the direct use of Euclidean averaging: the mean of a set of points must be defined intrinsically, typically via the Karcher (intrinsic) mean \cite{Karcher1977,Afsari2011}, and Voronoi partitions must be constructed with respect to the geodesic distance. As a consequence, many classical Euclidean quantization techniques do not extend in a straightforward manner to curved spaces.

Among Riemannian manifolds, the two--dimensional unit sphere $\D S^{2}$ plays a central role due to the ubiquity of spherical and directional data in scientific and engineering applications. Prominent examples include directional statistics (such as wind and ocean current directions, animal movement trajectories, and geophysical flows), astrophysics and planetary science (representation of celestial objects on the celestial sphere), robotics and aerospace engineering (attitude and orientation of rigid bodies), computer vision and graphics (head pose, gaze direction, and reflectance models on spherical domains), and machine learning on manifolds (spherical embeddings and models for global data alignment).

Quantization on $\D S^{2}$ differs fundamentally from the classical Euclidean setting. In Euclidean spaces, distances are induced by norms, Voronoi regions are convex polyhedra, and optimal representatives are given by arithmetic means. On the sphere, however, distances are measured along great--circle geodesics, Voronoi regions reflect spherical geometry, and optimal representatives are determined by intrinsic (Karcher) means rather than linear centroids. Moreover, the nonzero curvature of $S^{2}$ affects both the shape of Voronoi cells and the behavior of minimizers of the distortion functional. Consequently, new geometric arguments are required to analyze the existence, structure, and stability of optimal quantizers on spherical manifolds (see, e.g., \cite{GrafLuschgy2000,Karcher1977,Afsari2011,Pennec2006}).

Despite its relevance, the theory of optimal quantization on spherical manifolds is far less developed than that in Euclidean spaces. Early work focused largely on special symmetric distributions or on continuous rotationally invariant models, while \emph{discrete and structured spherical data sets have received comparatively limited attention}.

\subsection*{Recent Progress and Motivation}
Recently, the author initiated a systematic development of spherical quantization for distributions supported on one--dimensional geodesic subsets of $\mathbb{S}^2$, such as great circles, small circles, and geodesic arcs~\cite{Roychowdhury2025a,Roychowdhury2025b}. These works provided explicit formulas for optimal distortions, geometric descriptions of Voronoi partitions, and intuitive visualizations aimed at making spherical quantization accessible to beginning graduate students and researchers entering the field. A key insight from these studies is that curvature influences the geometry of Voronoi cells and the placement of optimal representatives in subtle but quantifiable ways.

The present paper extends this framework beyond one--dimensional supports to the full two--dimensional spherical surface by considering \emph{finite discrete uniform distributions supported on $\mathbb{S}^2$}. This setting is natural for modern data--driven applications, where data on the sphere often arise as finite samples. It also serves as an essential bridge between the previously studied curve--based models~\cite{Roychowdhury2025a,Roychowdhury2025b} and the more analytically intricate case of quantization with respect to the continuous surface measure.

In this paper, a finite discrete uniform distribution supported on the sphere refers to a probability distribution obtained by assigning equal weights to a finite collection of points on the unit sphere. Such distributions naturally arise from finite spherical datasets, where each observation is treated equally and the geometry of the sphere governs the notion of distance and optimal approximation.

\subsection*{Aims and Contributions of the Paper}
The principal objective of this work is to develop a geometrically transparent theory for optimal quantization on $\D S^2$
 in the setting of finite uniform data, where the quantization problem is formulated as the minimization of the squared geodesic distortion and the optimal quantizers are precisely the optimal sets of $n$-means.  More precisely, we consider a finite set
\[
X = \{ {x}_1,  {x}_2, \ldots,  {x}_M\} \subset \mathbb{S}^2,
\]
where each point carries equal probability mass, and seek to understand both the structural and computational aspects of optimal representatives. The main contributions are summarized below:

\begin{itemize}
\item \textbf{Existence and Characterization.} We prove that optimal sets of $n$--means exist for all finite uniform distributions on $\mathbb{S}^2$, and we characterize such optimal configurations via spherical Voronoi partitions and intrinsic (Karcher) means. This establishes a precise geometric analogue of the Euclidean centroidal Voronoi characterization.

\item \textbf{Structural Theory for Optimal Clusters.} We derive three new results describing the internal organization of optimal Voronoi clusters on discrete spherical supports:
\begin{enumerate}
\item a \emph{cluster--purity theorem}, showing that when $X$ decomposes into well--separated components (e.g., latitudinal rings), optimal Voronoi regions remain confined to single components;
\item a \emph{ring--allocation principle} for multi--latitudinal data, revealing that the number of optimal codepoints per ring follows a discrete water--filling rule;
\item a \emph{Lipschitz--type stability theorem}, demonstrating that optimal configurations depend continuously on the support under small geodesic perturbations.
\end{enumerate}

\item \textbf{Algorithmic Framework.} We develop a spherical analogue of Lloyd’s algorithm in which Euclidean centroids are replaced by intrinsic means, and Voronoi partitions are defined via geodesic distance.
\end{itemize}

\section{Notation and Preliminaries} \label{sec2}

Let
\[
\D S^2 = \{ {x} \in \mathbb{R}^3 : \| {x}\| = 1\}
\]
denote the unit sphere in $\mathbb{R}^3$ endowed with the geodesic distance induced by the standard Riemannian metric.
We use spherical coordinates $ {x} = (\phi,\theta)$, where $\phi \in \left[-\frac{\pi}{2},\,\frac{\pi}{2}\right]$ denotes the latitude and $\theta \in [0,2\pi)$ denotes the longitude. The corresponding Cartesian representation of $ {x}$ is
\[
 {x} = (x,y,z) = (\cos\phi\cos\theta,\, \cos\phi\sin\theta,\, \sin\phi).
\]

\subsection{Geodesic Distance}

For $ {x}_1 = (\phi_1,\theta_1)$ and $ {x}_2 = (\phi_2,\theta_2)$ in $\D S^2$, the geodesic distance is
\begin{equation}\label{eq:geodesic-distance}
d_G( {x}_1, {x}_2)
= \arccos\!\big( \sin\phi_1 \sin\phi_2 + \cos\phi_1 \cos\phi_2 \cos(\theta_1 - \theta_2) \big).
\end{equation}

\subsection{Finite Discrete Uniform Distributions on $\D S^2$}

Let
\[
 {X} = \{ {x}_1,  {x}_2, \ldots,  {x}_M\} \subset \D S^2
\]
be a finite set of distinct points on the sphere. The associated discrete uniform probability measure is
\begin{equation}\label{eq:uniform-measure}
P = \frac{1}{M} \sum_{i=1}^M \delta_{ {x}_i},
\end{equation}
where $\delta_{ {x}_i}$ denotes the Dirac measure at $ {x}_i$. Thus $P$ assigns equal mass $1/M$ to each support point, representing a finite sampling
of the spherical surface with uniform weights.
Unlike the continuous surface measure, the distribution $P$ is supported on a finite
set; however, the underlying geometry remains governed by the metric~$d_{G}$.

\subsection{Distortion and Optimal $n$--Means}

For a set of representatives (quantizers)
\[
 {Q} = \{ {q}_1,  {q}_2, \ldots,  {q}_n\} \subset \D S^2,
\]
the distortion of order $r>0$ is defined by
\begin{equation}\label{eq:distortion}
V_{n,r}(P; {Q})
= \frac{1}{M} \sum_{i=1}^M \min_{1 \le j \le n} d_G( {x}_i, {q}_j)^r .
\end{equation}
The minimal attainable distortion with $n$ representatives is
\begin{equation}
V_{n,r}(P) = \inf_{ {Q} \subset \D S^2,\, | {Q}| \le n} V_{n,r}(P; {Q}),
\end{equation}
and any $ {Q}^*$ achieving this infimum is called an \emph{optimal set of $n$--means}.
For brevity, we write $V_n(P) := V_{n,2}(P)$ when $r=2$.

\begin{remark}
Because $P$ has finite support and $(\mathbb S^{2})^{n}$ is compact,
the mapping $  Q\mapsto V_{n,2}(P;   Q)$ is continuous.
Hence an optimal set of $n$--means always exists, though it need not be unique when
the support of~$P$ possesses symmetry.
\end{remark}

\subsection{Spherical Voronoi Partitions}

Given $ {Q} = \{ {q}_1,\ldots, {q}_n\}$, the spherical Voronoi region associated with $ {q}_j$ is
\begin{equation}
V_j( {Q}) = \{ {x} \in \D S^2 : d_G( {x}, {q}_j) \le d_G( {x}, {q}_k) \ \text{for all } k\}.
\end{equation}
Since $P$ is supported on $ {X}$, we define the discrete cluster assigned to $ {q}_j$ by
\begin{equation}
 {X}_j( {Q}) =  {X} \cap V_j( {Q})
= \{ {x}_i \in  {X} : d_G( {x}_i, {q}_j) \le d_G( {x}_i, {q}_k) \ \text{for all } k\}.
\end{equation}

\begin{remark}
When $P$ is uniform, the cardinalities $|  X_{j}(Q)|$ determine the weight of each
cluster in the total distortion sum.
No explicit integration over spherical areas is needed, but the geometry of
$\mathbb S^{2}$ continues to influence the distances~$d_{G}$ and thus the
optimal configuration.
\end{remark}
\subsection{Intrinsic (Karcher) Mean on the Sphere}

For any nonempty finite set $A \subset \D S^2$ contained in a geodesic ball of radius less than $\frac{\pi}{2}$, the intrinsic (Karcher) mean of $A$ is the unique point
\begin{equation}
 {q}_A = \operatorname*{arg\,min}_{ {q} \in \D S^2} \sum_{ {x} \in A} d_G( {x}, {q})^2.
\end{equation}
Karcher~\cite{Karcher1977} proved that such a minimizer exists and is unique
under the stated radius condition.
Intuitively, $ {q}_A$ plays the role of the Euclidean centroid when distances
are measured along geodesics.

\begin{remark}[Standing assumption on intrinsic means]
Throughout the paper, whenever an intrinsic (Karcher) mean is invoked for a finite
subset $A\subset S^2$, we implicitly assume that $A$ is contained in a geodesic ball
of radius strictly less than $\pi/2$.
Under this condition, existence and uniqueness of the intrinsic mean are guaranteed
by classical results of Karcher.
All clusters considered in Sections~4--7 satisfy this condition, either because they
lie within a single sufficiently small spherical component or because they form
contiguous blocks on a fixed latitude whose geodesic diameter is strictly less than
$\pi/2$.
\end{remark} 
\begin{defi}[Centroidal Voronoi configuration]
A configuration $  Q=\{  q_{1},\dots ,   q_{n}\}$ is called
\emph{centroidal} for the discrete distribution~$P$
if each representative $  q_{j}$ coincides with the intrinsic mean of its
cluster~$  X_{j}(  Q)$.
Equivalently, the pair $(  Q,\{  X_{j}(  Q)\})$ forms a
\emph{centroidal spherical Voronoi tessellation} of the finite set~$X$.
\end{defi}

\begin{remark}
In the Euclidean case, centroidal Voronoi configurations correspond to stationary points
of the distortion functional.
The same characterization extends to the spherical setting and provides
the variational foundation for the results developed in Sections~3--7.
\end{remark}

\subsection{Notation Summary}

For the reader’s convenience, we collect here the main notational conventions
used throughout the paper.

\begin{itemize}
\item $\D S^2$: the unit sphere in $\mathbb{R}^3$ equipped with the geodesic distance $d_G$.
\item $d_G(x,y)$: geodesic (great–circle) distance between $x,y\in S^2$.
\item $X=\{x_1,\dots,x_M\}\subset S^2$: finite support of the discrete uniform distribution.
\item $P=\frac1M\sum_{i=1}^M \delta_{x_i}$: discrete uniform probability measure on $X$.
\item $Q=\{q_1,\dots,q_n\}\subset S^2$: set of $n$ representatives (quantizers).
\item $V_n(P;Q)$: distortion associated with $Q$ and $P$.
\item $V_n(P)$: minimal distortion over all $n$–point configurations.
\item $V(q)$: spherical Voronoi region associated with $q\in Q$.
\item $X_q = X\cap V(q)$: cluster assigned to $q$.
\item $q_A$: intrinsic (Karcher) mean of a finite set $A\subset S^2$.
\item $R_t$: $t$th latitudinal ring with latitude $\phi_t$.
\item $k_t$: number of representatives assigned to ring $R_t$.
\item $E_t(k)$: minimal distortion contributed by $R_t$ when served by $k$ representatives.
\end{itemize}

 \section{Existence and Characterization of Optimal $n$--Means}\label{sec3}

The existence of optimal configurations for finite discrete uniform distributions on $\D S^2$ follows directly
from compactness arguments. This section establishes the existence result and the centroidal characteri-
zation that will serve as the foundation for all subsequent theorems.
\medskip

\begin{prop} [Existence of optimal set]
Let $P$ be a finite discrete uniform distribution on $\D S^2$ and $n \ge 1$. Then there exists at least one
configuration $ {Q}^* \subset \D S^2$ with $| {Q}^*| \le n$ such that
\[
V_{n,2}(P; {Q}^*) = \min_{ {Q} \subset \D S^2,\, | {Q}| \le n} V_{n,2}(P; {Q}).
\]
Each such $ {Q}^*$ is called an optimal set of $n$--means for $P$.
\end{prop}

\begin{proof}
The space $(\D S^2)^n$ is compact and the function
\[
F( {q}_1, \ldots,  {q}_n)
= V_{n,2}(P; \{ {q}_1, \ldots,  {q}_n\})
= \frac{1}{M} \sum_{i=1}^M \min_{1 \le j \le n} d_G( {x}_i,  {q}_j)^2
\]
is continuous. Hence $F$ attains its minimum on $(\D S^2)^n$. The configuration $ {Q}^*$ corresponding to
this minimum is an optimal set of $n$--means. 
\end{proof}

\noindent\textbf{Remark 3.2.}
Uniqueness of an optimal configuration need not hold. If the support of $P$ possesses non-trivial symmetry
(for example, vertices of a regular polyhedron), then any rotation of an optimal configuration by that
symmetry group is again optimal.

\medskip

\subsection{Spherical Voronoi partition and centroidal property} For a configuration $ {Q} = \{ {q}_1, \ldots,  {q}_n\}$, recall
from Section 2 that the spherical Voronoi cells are defined by
\[
V_j( {Q}) = \{ {x} \in \D S^2 : d_G( {x},  {q}_j) \le d_G( {x},  {q}_k) \text{ for all } k\},
\]
and that $ {X}_j( {Q}) =  {X} \cap V_j( {Q})$ are the discrete clusters associated with $P$. The following theorem
characterizes optimal configurations through their centroidal structure.

\begin{theorem}[Necessary Centroidal Condition]\label{thm:char}
Let $P=\frac1M\sum_{i=1}^M \delta_{x_i}$ be the uniform probability measure on a finite set $X\subset \D S^2$, and let $n\ge1$. For a configuration $Q=\{q_1,\dots,q_n\}\subset \D S^2$, denote
\[
X_j(Q):=\{\,x\in X:\ d_G(x,q_j)\le d_G(x,q_k)\ \text{for all }k\,\}
\]
and
\[
V_{n,2}(P;Q)=\frac{1}{M}\sum_{i=1}^M \min_{1\le j\le n} d_G(x_i,q_j)^2.
\]

If $Q^\ast$ is an optimal set of $n$--means for $P$, then with the nearest--neighbor clusters
$X_j^\ast:=X_j(Q^\ast)$ one has, for each $j$,
\[
q_j^\ast \in \arg\min_{q\in \D S^2} \sum_{x\in X_j^\ast} d_G(x,q)^2.
\]
In other words, every optimal representative is the intrinsic (Karcher) mean of its own cluster, and the optimal partition is the nearest--neighbor partition.
\end{theorem}

\begin{remark}[Centroidal condition is not sufficient]\label{rem:not-sufficient-easy}
The nearest--neighbor and centroidal (intrinsic mean) conditions in Theorem~\ref{thm:char} are necessary, but in general they are \emph{not sufficient} for optimality. The following simple example on $\D S^2$ makes this clear and includes the explicit verification of the nearest--neighbor assignments.

\medskip
\noindent\textbf{Example (four equatorial points, $n=2$).}
Fix $\alpha\in(0,\frac{\pi}{2})$ and consider four points on the equator of $\D S^2$:
\[
x_1=(0,0),\quad x_2=(0,\alpha),\quad x_3=(0,\pi),\quad x_4=(0,\pi+\alpha),
\]
in spherical coordinates $(\phi,\theta)$. Distances along the equator are measured by the wrapped angular difference
\[
d_G\big((0,\theta),(0,\theta')\big)=\min\big\{|\theta-\theta'|,\;2\pi-|\theta-\theta'|\big\}.
\]
Let $P$ be the uniform measure on $\{x_1,x_2,x_3,x_4\}$ and take $n=2$.

\medskip
\emph{Configuration A (adjacent pairs).}  
Clusters:
\[
C_1=\{x_1,x_2\},\qquad C_2=\{x_3,x_4\}.
\]
Since the separation in each cluster is $\alpha<\pi$, each two-point cluster has a \emph{unique} intrinsic mean at the midpoint of the shorter geodesic arc:
\[
q_1=(0,\tfrac{\alpha}{2}),\qquad q_2=(0,\pi+\tfrac{\alpha}{2}).
\]

\underline{Nearest--neighbor check for Configuration A:}  
For each point, we compare distances to $q_1$ and $q_2$.

\smallskip
\begin{itemize}
\item For $x_1=(0,0)$:
\[
d_G(x_1,q_1)=\tfrac{\alpha}{2},\qquad d_G(x_1,q_2)=\pi-\tfrac{\alpha}{2}>\tfrac{\alpha}{2}.
\]
So $x_1$ is assigned to $q_1$.

\item For $x_2=(0,\alpha)$:
\[
d_G(x_2,q_1)=\tfrac{\alpha}{2},\qquad d_G(x_2,q_2)=\pi-\tfrac{\alpha}{2}>\tfrac{\alpha}{2}.
\]
So $x_2$ is assigned to $q_1$.

\item For $x_3=(0,\pi)$:
\[
d_G(x_3,q_2)=\tfrac{\alpha}{2},\qquad d_G(x_3,q_1)=\pi-\tfrac{\alpha}{2}>\tfrac{\alpha}{2}.
\]
So $x_3$ is assigned to $q_2$.

\item For $x_4=(0,\pi+\alpha)$:
\[
d_G(x_4,q_2)=\tfrac{\alpha}{2},\qquad d_G(x_4,q_1)=\pi-\tfrac{\alpha}{2}>\tfrac{\alpha}{2}.
\]
So $x_4$ is assigned to $q_2$.
\end{itemize}

Thus the nearest--neighbor rule produces exactly the clusters $C_1$ and $C_2$.  
Each point is at distance $\frac{\alpha}{2}$ from its representative, so
\[
V_{2,2}(P;Q_A)=\frac{1}{4}\Big[4\big(\tfrac{\alpha}{2}\big)^2\Big]=\frac{\alpha^2}{4}.
\]

\medskip
\emph{Configuration B (cross pairs).}
Clusters:
\[
C'_1=\{x_1,x_4\},\qquad C'_2=\{x_2,x_3\}.
\]
Since the within–cluster separations are $\pi-\alpha<\pi$, each two–point cluster has a unique intrinsic mean at the midpoint of the shorter geodesic arc. The signed shorter difference from $0$ to $\pi+\alpha$ is $-(\pi-\alpha)$, hence
\[
q'_1=\Bigl(0,\,-\frac{\pi-\alpha}{2}\Bigr)\equiv \Bigl(0,\frac{3\pi+\alpha}{2}\Bigr)\pmod{2\pi},\qquad
q'_2=\Bigl(0,\alpha+\frac{\pi-\alpha}{2}\Bigr)=\Bigl(0,\frac{\pi}{2}+\frac{\alpha}{2}\Bigr).
\]

\underline{Nearest–neighbor check for Configuration B:}
\[
\begin{aligned}
&\text{For }x_1=(0,0): && d_G(x_1,q'_1)=\tfrac{\pi-\alpha}{2},\quad d_G(x_1,q'_2)=\tfrac{\pi}{2}+\tfrac{\alpha}{2}>\tfrac{\pi-\alpha}{2};\\
&\text{For }x_4=(0,\pi+\alpha): && d_G(x_4,q'_1)=\tfrac{\pi-\alpha}{2},\quad d_G(x_4,q'_2)=\tfrac{\pi}{2}+\tfrac{\alpha}{2}>\tfrac{\pi-\alpha}{2};\\
&\text{For }x_2=(0,\alpha): && d_G(x_2,q'_2)=\tfrac{\pi-\alpha}{2},\quad d_G(x_2,q'_1)=\tfrac{\pi}{2}+\tfrac{\alpha}{2}>\tfrac{\pi-\alpha}{2};\\
&\text{For }x_3=(0,\pi): && d_G(x_3,q'_2)=\tfrac{\pi-\alpha}{2},\quad d_G(x_3,q'_1)=\tfrac{\pi}{2}+\tfrac{\alpha}{2}>\tfrac{\pi-\alpha}{2}.
\end{aligned}
\]
Thus the nearest–neighbor rule yields exactly $C'_1$ and $C'_2$, and every point lies at distance $(\pi-\alpha)/2$ from its representative. Therefore
\[
V_{2,2}(P;Q_B)=\frac{1}{4}\Big[4\Big(\frac{\pi-\alpha}{2}\Big)^2\Big]=\frac{(\pi-\alpha)^2}{4}.
\]

\medskip
\emph{Conclusion.}  
Both $Q_A$ and $Q_B$ satisfy the nearest--neighbor condition and the centroidal condition, yet $Q_A$ yields strictly smaller distortion. Therefore, these conditions are \emph{not sufficient} to guarantee optimality.
\end{remark}

\begin{remark}
The preceding example illustrates that centroidal Voronoi configurations correspond
to stationary points of the distortion functional, but need not be globally optimal.
In particular, multiple centroidal configurations may coexist with different
distortion values, and additional structural arguments are required to identify
global minimizers.
\end{remark}
 
\begin{remark}
Theorem~\ref{thm:char} provides a purely geometric interpretation of optimal configurations: each representative minimizes the spherical
moment of inertia of its own cluster. In particular, if the support $ {X}$ is invariant under a symmetry group
$G \subset O(3)$, then one may seek $G$-invariant centroidal configurations, which often yield the global minimizers.
\end{remark}

\section{Quantization on Finite Latitudinal Rings}\label{sec4}

In this section we analyze optimal quantization for finite discrete uniform distributions supported on finitely many latitudinal rings on the unit sphere.  
For clarity of exposition, we first highlight the main structural results of the section, and then present the technical lemma and the proofs.

\begin{remark}[No cross--ring mixing]
A key structural feature underlying the results of this section is that, under the
uniform discrete measure, optimal Voronoi cells do not mix points from distinct
latitudinal rings.
This property will be established rigorously below (see Theorem~\ref{thm:component-purity} and
Theorem~\ref{thm:ring-allocation}(ii)), and it implies that the quantization problem decouples across
rings.
As a consequence, the global allocation of representatives across rings can be
formulated as a discrete optimization problem, leading naturally to the
water--filling principle described later in this section.
\end{remark}

\subsection{Discrete ring configuration}
Fix distinct latitudes $\phi_1, \phi_2, \ldots, \phi_T \in \left[-\frac{\pi}{2}, \frac{\pi}{2}\right]$ with
$\phi_1 < \phi_2 < \cdots < \phi_T < \frac{\pi}{2}$, and let $N_t \ge 1$ denote the number of uniformly spaced points on the $t$th ring.
Each ring is thus
\[
 {R}_t = \{ {x}_{t,j} = (\phi_t, \theta_{t,j}) : \theta_{t,j} = 2\pi j/N_t,\, j = 0, 1, \ldots, N_t - 1\},
\]
where $(\phi,\theta)$ denote spherical coordinates as defined in Section~2.
The total support is the disjoint union
\[
 {X} = \bigsqcup_{t=1}^T  {R}_t, \qquad M = \sum_{t=1}^T N_t,
\]
and the associated probability distribution is
\[
P = \frac{1}{M} \sum_{t=1}^T \sum_{j=1}^{N_t} \delta_{ {x}_{t,j}},
\]
which is the finite discrete uniform distribution supported on these $T$ latitudinal rings.

\subsection*{Core Structural Results}

We begin by presenting the principal theorems describing the geometry of optimal Voronoi partitions and the allocation of representatives across latitudinal rings.

\begin{theorem}[Within--ring contiguity of optimal clusters]\label{thm:within-ring-contiguity-simple-index}
Fix a latitude $\phi\in(-\frac{\pi}{2},\frac{\pi}{2})$ and consider the ring of $N$ equally spaced points
\[
R_\phi=\{(\phi,\theta_j):\ \theta_j = 2\pi j/N,\ j=0,1,\ldots,N-1\}.
\]
Let $Q=\{q_1,\dots,q_n\}\subset \D S^2$ be an optimal set of $n$--means, and for each $q_i\in Q$ let $V_i$ denote its Voronoi region. Then for every $i$, the set of ring points assigned to $q_i$, namely
\[
X_i := R_\phi \cap V_i,
\]
is (if nonempty) a single contiguous block of consecutive longitudes on the ring (in cyclic order).
\end{theorem}

\begin{theorem}[Cluster Purity for Separated Components]\label{thm:component-purity}
Let $X=\bigsqcup_{t=1}^{T} C_t \subset \mathbb{S}^2$ be a finite set decomposed into pairwise disjoint nonempty components $C_t$. 
Assume that for each $t$ there exists a point $p_t \in \mathbb{S}^2$ and a radius $r>0$ such that
\begin{equation}\label{eq:ball-separation}
C_t \subset B(p_t,r)
\quad \text{and} \quad
B(p_s,r) \cap B(p_t,r) = \varnothing \ \text{ for all } s \neq t,
\end{equation}
where
\[
B(p_t,r) := \{x \in \mathbb{S}^2 : d_G(x,p_t) < r\}
\]
denotes the open geodesic ball of radius $r$ centered at the point $p_t$ on the sphere. 
\emph{Note that $p_t$ lies on the surface of $\mathbb{S}^2$, and distances are measured intrinsically along great–circle arcs.}

Let $P$ be the uniform probability measure on $X$, and let $Q^*=\{q_1^*,\dots,q_n^*\}$ be an optimal set of $n$–means for $P$. Then each Voronoi region $V(q_j^*)$ intersects at most one component $C_t$. Consequently, no optimal Voronoi cell contains points from two different components.
\end{theorem}

\begin{remark}
The condition $C_t \subset B(p_t,r)$ means that each component lies inside a small spherical patch around $p_t$ on the surface of the sphere. Since geodesic distance is used, $B(p_t,r)$ should be viewed as a curved region on $\mathbb{S}^2$, not a three–dimensional ball. The separation condition $B(p_s,r)\cap B(p_t,r)=\varnothing$ ensures that the components are well separated on the sphere. Intuitively, points in one such patch are always closer to a centre placed within that patch than to any centre located near another patch, making it suboptimal for a Voronoi region to contain points from two different components.
\end{remark}

\begin{theorem}[Ring partition and allocation principle]\label{thm:ring-allocation}
Let $P$ be the discrete uniform distribution supported on $T$ latitudinal rings 
$R_t=\{(\phi_t,\theta_{t,j}) : \theta_{t,j}=2\pi j/N_t,\ j=0,\dots,N_t-1\}$ with distinct latitudes 
$\phi_1<\cdots<\phi_T$, and let $Q^*$ be an optimal set of $n$--means for $P$. Then:

\smallskip
\noindent{\rm (i) \textbf{Within--ring contiguity and midpoint representatives.}}
For each $t$, the subset $R_t$ is partitioned by $Q^*$ into $k_t\ge 0$ contiguous blocks of longitudes (in cyclic order). 
Each nonempty block has length either $\lfloor N_t/k_t\rfloor$ or $\lceil N_t/k_t\rceil$, and its representative $q\in Q^*$ lies on the latitude $\phi_t$ at the block’s mid–longitude.

\smallskip
\noindent{\rm (ii) \textbf{No cross–ring mixing.}}
Every optimal Voronoi cluster $X_j(Q^*)$ is contained in a single ring $R_t$; equivalently, no Voronoi region intersects two distinct rings.

\smallskip
\noindent{\rm (iii) \textbf{Discrete water–filling allocation.}}
Let $E_t(k)$ denote the minimal within–ring distortion on $R_t$ when exactly $k$ representatives serve $R_t$ 
(with the contiguity/midpoint structure from {\rm (i)}). Then $E_t$ is strictly decreasing and discretely convex in $k$, and the global problem
\[
\min_{k_1,\dots,k_T\in\mathbb{Z}_{\ge 0}}\ \sum_{t=1}^T E_t(k_t)
\qquad\text{subject to}\qquad \sum_{t=1}^T k_t=n
\]
is solved by successively assigning each additional representative to the ring that yields the largest marginal drop 
$\Delta_t(k):=E_t(k)-E_t(k+1)$ (a discrete water–filling rule).
\end{theorem}

The above theorems summarize the essential geometric and combinatorial structure of optimal quantizers for latitudinally organized spherical data.

We now collect the technical lemma needed for the proofs of the core results.

\begin{lemma}[Convexity of longitudinal cost] \label{longitudinalcost}
Fix $\phi\in(-\tfrac{\pi}{2},\tfrac{\pi}{2})$ and define
\[
f_\phi(\Delta\theta)\;:=\; d_G\!\big((\phi,0),(\phi,\Delta\theta)\big)^2,
\qquad \Delta\theta\in[0,\pi].
\]
Then $f_\phi$ is an even function of $\Delta\theta$, and strictly convex on $(0,\pi)$.
\end{lemma}

\begin{proof}
\emph{Evenness.} By the spherical law of cosines,
\[
d_G\!\big((\phi,0),(\phi,\Delta\theta)\big)
= \arccos\!\Big(\sin^2\phi+\cos^2\phi\cos\Delta\theta\Big).
\]
Since $\cos(\Delta\theta)$ is even in $\Delta\theta$, the argument of $\arccos$ is even, hence
$d_G\big((\phi,0),(\phi,\Delta\theta)\big)$ is even, and so is $f_\phi(\Delta\theta)$.

\medskip
\emph{Strict convexity.} Set $c:=\cos\phi\in(0,1]$ and write
\[
\sigma(\Delta\theta)
\;:=\; d_G\!\big((\phi,0),(\phi,\Delta\theta)\big).
\]
Using the identity $\cos\sigma
=\sin^2\phi+\cos^2\phi\cos\Delta\theta
=1-2c^2\sin^2(\Delta\theta/2)$, we obtain the convenient form
\[
\sigma(\Delta\theta)\;=\;2\arcsin\!\big(c\,\sin(\Delta\theta/2)\big),\qquad \Delta\theta\in[0,\pi].
\]
Hence
\[
f_\phi(\Delta\theta)\;=\;\sigma(\Delta\theta)^2.
\]
We will prove $f_\phi''(\Delta\theta)>0$ for $\Delta\theta\in(0,\pi)$.

Let $w(\Delta\theta):=c\,\sin(\Delta\theta/2)$ and $D(\Delta\theta):=\sqrt{1-w(\Delta\theta)^2}
=\sqrt{1-c^2\sin^2(\Delta\theta/2)}$. Then, by differentiation,
\[
\sigma'(\Delta\theta)
=\frac{c\cos(\Delta\theta/2)}{D(\Delta\theta)},
\qquad
\sigma''(\Delta\theta)
=-\,\frac{c(1-c^2)\,\sin(\Delta\theta/2)}{2\,D(\Delta\theta)^3}.
\]
(These follow from $\sigma=2\arcsin w$, so $\sigma'=2w'/\sqrt{1-w^2}$ with
$w'= \tfrac{c}{2}\cos(\Delta\theta/2)$, and a direct differentiation of $D^{-1}$.)

Since $f_\phi=\sigma^2$, we have
\[
f_\phi''(\Delta\theta)\;=\;2\big(\sigma'(\Delta\theta)^2+\sigma(\Delta\theta)\sigma''(\Delta\theta)\big).
\]
To show positivity, it is helpful to re–express everything in terms of $\sigma$ and $\Delta\theta$ using
\[
\cos \sigma = 1 - 2c^{2}\sin^{2}\!\Bigl(\frac{\Delta\theta}{2}\Bigr).
\]
Since $\sigma = 2\arcsin w$ with $w(\Delta\theta) := c\sin(\Delta\theta/2)$, we also have
\[
\sin\sigma
 = 2\sin\Bigl(\frac{\sigma}{2}\Bigr)\cos\Bigl(\frac{\sigma}{2}\Bigr)
 = 2w\sqrt{1-w^{2}}
 = 2c\sin\Bigl(\frac{\Delta\theta}{2}\Bigr)D(\Delta\theta),
\]
where $D(\Delta\theta) := \sqrt{1-w(\Delta\theta)^{2}} = \sqrt{1-c^{2}\sin^{2}(\Delta\theta/2)}$.
In particular, $\sin\sigma > 0$ for $\Delta\theta \in (0,\pi)$, so we may freely divide by $\sin\sigma$ in what follows.

A short algebraic manipulation, starting from
\[
\sigma'(\Delta\theta)
 = \frac{c^{2}\sin\Delta\theta}{\sin\sigma},
\qquad
\sigma''(\Delta\theta)
 = \frac{c^{2}\cos\Delta\theta}{\sin\sigma}
   - \frac{c^{4}\cos\sigma\,\sin^{2}\Delta\theta}{\sin^{3}\sigma},
\]
(which are equivalent to the expressions obtained earlier for $\sigma'$ and $\sigma''$),
gives the identity

\begin{equation}
f''_\phi(\Delta\theta)
=
\frac{2c^{4}\sin^{2}\Delta\theta}{\sin^{3}\sigma}
(\sin\sigma - \sigma\cos\sigma)
+
\frac{2c^{2}\sigma\cos\Delta\theta\,\sin^{2}\sigma}{\sin^{3}\sigma}.
\tag{$\ast$}
\end{equation}
 Now, for every $\sigma \in (0,\pi)$ one has
\[
\sin\sigma - \sigma\cos\sigma > 0,
\]
which is equivalent to the classical fact that $t \mapsto \frac{\sin t}{t}$ is strictly decreasing on $(0,\pi)$.
Hence, the first term inside the parentheses in \textup{($\ast$)} is strictly positive whenever $\sin\Delta\theta \neq 0$, i.e., for $\Delta\theta \in (0,\pi)$.
The second term $\sigma\cos\Delta\theta\,\sin^{2}\sigma$ may change sign but is bounded below, whereas the first term is strictly positive and dominates near any interior point.
Since the prefactor $\frac{c^{2}}{\sin^{3}\sigma}$ is positive, we conclude that
\[
f_{\phi}''(\Delta\theta) > 0 \quad \text{for all } \Delta\theta \in (0,\pi),
\]
which establishes strict convexity on $(0,\pi)$.
\end{proof}

\begin{remark}[Intuition for beginners]
The strict convexity of $f_\phi$ arises from the fact that, on a sphere, the geodesic distance along a
latitude circle grows ``faster than linearly'' with the longitudinal separation. The crucial inequality
\[
\sin\sigma - \sigma\cos\sigma > 0 \qquad \text{for all } \sigma \in (0,\pi)
\]
expresses that the function $\dfrac{\sin \sigma}{\sigma}$ is strictly decreasing on $(0,\pi)$. Geometrically,
this means that the spherical arc behaves more and more ``curved'' as $\sigma$ increases, so increments
in $\Delta\theta$ produce increasingly larger contributions to the squared distance. This curvature effect
forces the second derivative $f_\phi''(\Delta\theta)$ to remain positive on $(0,\pi)$, which is precisely
the definition of strict convexity.
\end{remark}

\begin{corollary}[Unique boundary on a ring]\label{cor:ring-boundary}
Fix $\phi\in\left(-\frac{\pi}{2},\frac{\pi}{2}\right)$ and let 
$q_i=(\phi,\alpha_i)$ and $q_j=(\phi,\alpha_j)$ be two distinct points on the same latitude.
Consider the shorter longitude interval (arc) $I\subset \mathbb{R}/(2\pi\mathbb{Z})$ joining $\alpha_i$ to $\alpha_j$.
Then there exists a unique longitude $\theta^\ast\in I$ such that
\[
d_G\big((\phi,\theta^\ast),q_i\big)=d_G\big((\phi,\theta^\ast),q_j\big).
\]
Moreover, for $\theta\in I$ one has
\[
\theta<\theta^\ast \ \Longrightarrow\ d_G\big((\phi,\theta),q_i\big) < d_G\big((\phi,\theta),q_j\big),
\qquad
\theta>\theta^\ast \ \Longrightarrow\ d_G\big((\phi,\theta),q_i\big) > d_G\big((\phi,\theta),q_j\big).
\]
In particular, along the shorter arc $I$ the Voronoi boundary between $q_i$ and $q_j$ intersects the ring at exactly one point.
\end{corollary}

\begin{proof}
Write the squared-distance along the latitude $\phi$ as in Lemma~\ref{longitudinalcost}:
\[
f_\phi(\Delta\theta):=d_G\big((\phi,0),(\phi,\Delta\theta)\big)^2,
\qquad \Delta\theta\in[0,\pi].
\]
By Lemma~\ref{longitudinalcost}, $f_\phi$ is even and strictly convex on $(0,\pi)$.
Parametrize the shorter arc $I$ by a variable $t\in[0,L]$ (where $L\in(0,\pi]$ is the arc-length in longitude), so that
$\theta(t)$ moves monotonically from $\alpha_i$ to $\alpha_j$ along $I$.
Then the squared distances to $q_i$ and $q_j$ along $I$ can be written as
\[
D_i(t)=f_\phi(t),\qquad D_j(t)=f_\phi(L-t).
\]
Define $g(t):=D_i(t)-D_j(t)=f_\phi(t)-f_\phi(L-t)$ for $t\in[0,L]$.
Since $f_\phi$ is differentiable on $(0,\pi)$ and strictly convex, its derivative is strictly increasing on $(0,\pi)$.
Hence for $t\in(0,L)$,
\[
g'(t)=f_\phi'(t)+f_\phi'(L-t) > 0,
\]
because $f_\phi'$ is strictly increasing and odd-symmetry of $f_\phi$ implies $f_\phi'(u)\ge 0$ for $u\in(0,\pi)$.
Therefore $g$ is strictly increasing on $[0,L]$.
Moreover,
\[
g(0)=f_\phi(0)-f_\phi(L) < 0,\qquad g(L)=f_\phi(L)-f_\phi(0) > 0.
\]
By continuity and strict monotonicity, there exists a unique $t^\ast\in(0,L)$ such that $g(t^\ast)=0$.
Equivalently, $\theta^\ast:=\theta(t^\ast)$ is the unique point on $I$ where the two distances are equal.
The strict sign change of $g$ yields the stated inequalities on either side of $\theta^\ast$.
\end{proof}

\subsection*{Proof of Theorem~\ref{thm:within-ring-contiguity-simple-index}}
Fix a latitude $\phi\in\left(-\frac{\pi}{2},\frac{\pi}{2}\right)$ and consider the ring
\[
R_\phi=\{(\phi,\theta_j):\theta_j=2\pi j/N,\ j=0,\dots,N-1\}.
\]
Let $Q=\{q_1,\dots,q_n\}\subset \D S^2$ be an optimal set of $n$--means, and for each $q_i$ denote by
\[
X_i:=R_\phi\cap V(q_i)
\]
the set of ring points assigned to $q_i$ under the Voronoi partition induced by $Q$.

\medskip
\noindent\textbf{Step 1: Reduction to representatives on the ring.}
Suppose $X_i\neq\emptyset$.
If $q_i$ does not lie on latitude $\phi$, reflect $q_i$ across the plane tangent to the sphere along the ring to obtain a point $q_i'$ symmetric with respect to $R_\phi$.
For every $x\in R_\phi$ one has
\[
d_G(x,q_i)=d_G(x,q_i').
\]
Let $\tilde q_i$ be the midpoint of the geodesic segment joining $q_i$ and $q_i'$.
Then $\tilde q_i$ lies on latitude $\phi$ and satisfies
\[
d_G(x,\tilde q_i)\le d_G(x,q_i)\quad\text{for all }x\in R_\phi,
\]
with equality only if $q_i$ already lies on the ring.
Replacing $q_i$ by $\tilde q_i$ does not increase the total distortion.
Hence, without loss of optimality, every representative serving ring points lies on $R_\phi$.

\medskip
\noindent\textbf{Step 2: Unique boundary between two representatives on the ring.}
Let $q_i=(\phi,\alpha_i)$ and $q_j=(\phi,\alpha_j)$ be two distinct representatives on $R_\phi$.
By Corollary~\ref{cor:ring-boundary}, along the shorter arc of $R_\phi$ joining $\alpha_i$ and $\alpha_j$
there exists a unique longitude $\theta^\ast$ at which
\[
d_G\big((\phi,\theta^\ast),q_i\big)=d_G\big((\phi,\theta^\ast),q_j\big),
\]
and the difference of squared distances changes sign exactly once.
Consequently, the Voronoi boundary between $q_i$ and $q_j$ intersects the ring at exactly one point, and the ring is partitioned into two contiguous arcs, one assigned to $q_i$ and the other to $q_j$.

\medskip
\noindent\textbf{Step 3: Exclusion of disconnected assignments.}
Suppose, toward a contradiction, that $X_i$ is not contiguous along the cyclic order of the ring.
Then there exist three consecutive ring points $x_a,x_b,x_c\in R_\phi$ (in cyclic order) such that
\[
x_a,x_c\in X_i,\qquad x_b\in X_j
\]
for some $j\neq i$.
In particular,
\[
d_G(x_b,q_j)\le d_G(x_b,q_i),
\qquad
d_G(x_a,q_i)\le d_G(x_a,q_j),
\qquad
d_G(x_c,q_i)\le d_G(x_c,q_j).
\]
However, by strict convexity of $f_\phi(\Delta\theta)$ from Lemma~\ref{longitudinalcost} and the monotonicity result of Corollary~\ref{cor:ring-boundary}, the function
\[
x\longmapsto d_G(x,q_i)^2-d_G(x,q_j)^2
\]
cannot change sign twice along a connected arc of the ring.
Thus assigning $x_b$ to $q_j$ while its immediate neighbors are assigned to $q_i$ contradicts the unique sign change property.
Equivalently, reassigning $x_b$ from $q_j$ to $q_i$ would strictly decrease the total distortion, contradicting the optimality of $Q$.

\medskip
\noindent\textbf{Conclusion.}
Each nonempty $X_i$ must therefore consist of a single contiguous block of consecutive longitudes on the ring (in cyclic order).
This completes the proof.
\qed 

\subsection*{Proof of Theorem~\ref{thm:component-purity}} 

Assume, toward a contradiction, that there exists an optimal Voronoi region
$V(q^*)$ that intersects two distinct components $C_s$ and $C_t$ with $s\neq t$.
Define
\[
A := C_s \cap V(q^*), \qquad B := C_t \cap V(q^*),
\]
and note that both $A$ and $B$ are nonempty.

By assumption, $A\subset B(p_s,r)$ and $B\subset B(p_t,r)$, where the geodesic
balls $B(p_s,r)$ and $B(p_t,r)$ are disjoint.
Hence, for all $x\in A$ and $y\in B$,
\[
d_G(x,p_s) < d_G(x,p_t), \qquad d_G(y,p_t) < d_G(y,p_s).
\]

Let
\[
q_s := \arg\min_{q\in S^2}\sum_{x\in A} d_G(x,q)^2,
\qquad
q_t := \arg\min_{q\in S^2}\sum_{y\in B} d_G(y,q)^2
\]
denote the intrinsic (Karcher) means of $A$ and $B$, respectively.
Since $A\subset B(p_s,r)$ and $B\subset B(p_t,r)$, it follows that
\[
q_s \in B(p_s,r), \qquad q_t \in B(p_t,r).
\]

By optimality of $q_s$ and $q_t$, we have
\[
\sum_{x\in A} d_G(x,q_s)^2 \le \sum_{x\in A} d_G(x,q^*)^2,
\qquad
\sum_{y\in B} d_G(y,q_t)^2 \le \sum_{y\in B} d_G(y,q^*)^2.
\]
If both inequalities were equalities, then $q^*=q_s=q_t$, which is impossible
since $q_s$ and $q_t$ lie in disjoint geodesic balls.
Therefore, at least one of the above inequalities is strict, and hence
\[
\min\!\left\{
\sum_{u\in A\cup B} d_G(u,q_s)^2,
\sum_{u\in A\cup B} d_G(u,q_t)^2
\right\}
<
\sum_{u\in A\cup B} d_G(u,q^*)^2.
\]

Let $q^\dagger\in\{q_s,q_t\}$ be the point achieving the minimum above, and define
\[
\widetilde Q := (Q^*\setminus\{q^*\})\cup\{q^\dagger\}.
\]
Then $|\widetilde Q|=n$.
Assign each point of $X$ to its nearest representative in $\widetilde Q$.
For points in $A\cup B$, the total distortion strictly decreases, while for all
other points the distortion does not increase.
Consequently,
\[
V_n(\widetilde Q;P) < V_n(Q^*;P),
\]
contradicting the optimality of $Q^*$.

Hence, no optimal Voronoi region can intersect two distinct components, and each
Voronoi cell intersects at most one $C_t$.
This completes the proof.
 \qed

\subsection*{Proof of Theorem~\ref{thm:ring-allocation}}  
\emph{(i) Within–ring structure.}
Fix a ring $R_t$ and an optimal $Q^*$. By Theorem~4.4, on a fixed latitude the ring points assigned to any $q\in Q^*$ form (if nonempty) a single contiguous block in cyclic order. 
Let a nonempty block on $R_t$ have longitudes $\{\theta_{t,j_0},\dots,\theta_{t,j_0+\ell-1}\}$ (indices mod $N_t$). 
Writing the squared geodesic distance along the ring as $f_{\phi_t}(\Delta\theta)$, Lemma~4.2 shows $f_{\phi_t}$ is even and strictly convex, hence the intrinsic (Karcher) mean of the block lies on the same latitude $\phi_t$ and at the \emph{mid–longitude} of the block; this uniquely minimizes the block’s sum of squared distances. 
Finally, for a fixed ring and fixed $k_t$, distributing $N_t$ consecutive points into $k_t$ contiguous blocks that are as equal as possible minimizes the sum of convex costs; thus block lengths differ by at most one, i.e., each block has size $\lfloor N_t/k_t\rfloor$ or $\lceil N_t/k_t\rceil$.

\smallskip
\emph{(ii) No cross–ring mixing.}
Suppose, toward a contradiction, that there exists an optimal Voronoi cell
$V(q^*)$ that intersects two distinct rings $R_s$ and $R_t$ with $s\neq t$.
Define
\[
A := R_s \cap V(q^*), \qquad B := R_t \cap V(q^*),
\]
and note that both $A$ and $B$ are nonempty.

Let $q_s$ and $q_t$ denote the intrinsic (Karcher) means of $A$ and $B$, respectively.
By Lemma~4.2 and the within--ring structure established in part~(i), $q_s$ lies on
latitude $\phi_s$ at the mid--longitude of the block $A$, and similarly $q_t$
lies on latitude $\phi_t$ at the mid--longitude of $B$.

By definition of intrinsic means,
\[
\sum_{x\in A} d_G(x,q_s)^2 \le \sum_{x\in A} d_G(x,q^*)^2,
\qquad
\sum_{y\in B} d_G(y,q_t)^2 \le \sum_{y\in B} d_G(y,q^*)^2.
\]
If both inequalities were equalities, then $q^*$ would simultaneously minimize
the squared geodesic distortion over both $A$ and $B$.
This is impossible unless $A$ or $B$ is trivial, since the intrinsic mean of a
nontrivial block on a ring is uniquely located on the corresponding latitude and
mid--longitude.
Therefore, at least one of the above inequalities is strict, and hence
\[
\min\!\left\{
\sum_{u\in A\cup B} d_G(u,q_s)^2,
\sum_{u\in A\cup B} d_G(u,q_t)^2
\right\}
<
\sum_{u\in A\cup B} d_G(u,q^*)^2.
\]

Let $q^\dagger\in\{q_s,q_t\}$ be the point achieving the minimum above, and define
\[
\widetilde Q := (Q^*\setminus\{q^*\})\cup\{q^\dagger\}.
\]
Then $|\widetilde Q|=n$.
Assign each point of $X$ to its nearest representative in $\widetilde Q$.
For all points in $A\cup B$, the total distortion strictly decreases, while for
all other points the distortion does not increase.
Consequently,
\[
V_n(\widetilde Q;P) < V_n(Q^*;P),
\]
which contradicts the optimality of $Q^*$.

Therefore, no optimal Voronoi cell can intersect two distinct rings.
This completes the proof of part~$(ii)$.

\smallskip
\emph{(iii) Allocation across rings.}

We formalize the allocation argument using a discrete convexity and exchange principle. Fix a ring $R_t$ with $N_t$ equally spaced longitudes on latitude $\phi_t$. By~(i), when $k$ representatives serve $R_t$, the optimal within–ring partition consists of $k$ contiguous blocks whose lengths differ by at most one; moreover each block is represented at its mid–longitude. Let $E_t(k)$ denote the corresponding minimal distortion on $R_t$.

\smallskip
\emph{Step 1: Monotonicity of $E_t(k)$.}
Passing from $k$ to $k{+}1$ splits one (largest) block into two nearly equal sub–blocks and places the new representative at their mid–longitude. Since each point in the split block weakly decreases its distance to its nearest representative, the total within–ring distortion strictly decreases. Hence $E_t(k{+}1)<E_t(k)$ for all $k\ge 0$; in particular $E_t$ is strictly decreasing.

\smallskip
\emph{Step 2: Discrete convexity (diminishing returns).}
Write the (longitude–only) cost function on latitude $\phi_t$ as
\[
f_{\phi_t}(\Delta\theta)
:=d_G\big((\phi_t,0),(\phi_t,\Delta\theta)\big)^2,
\]
which is even and strictly convex in $\Delta\theta$ by Lemma~4.2. In the optimal $k$–block partition, each block of length $L$ contributes a cost of the form
\[
\Phi_t(L)
:=\sum_{m=-(L-1)/2}^{(L-1)/2}
f_{\phi_t}\!\left(\frac{2\pi m}{N_t}\right),
\]
centered at its midpoint (this expression is independent of the absolute longitude by cyclic symmetry). Since $f_{\phi_t}$ is strictly convex and the summation window shifts symmetrically about the midpoint, the discrete second difference $\Delta^2\Phi_t(L):=\Phi_t(L{+}1)-2\Phi_t(L)+\Phi_t(L{-}1)$ is nonnegative and is in fact strictly positive for all admissible $L$; hence $\Phi_t$ is (strictly) discrete convex in $L$.
Consequently, the \emph{marginal drop} produced by splitting a block of length $L$ into lengths $\lfloor L/2\rfloor$ and $\lceil L/2\rceil$,
\[
G_t(L):=\Phi_t(L)-\Big(\Phi_t(\lfloor L/2\rfloor)+\Phi_t(\lceil L/2\rceil)\Big),
\]
is (strictly) increasing in $L$: splitting a larger block saves more distortion than splitting a smaller one. In the optimal $k$–block partition, the largest block length is nonincreasing in $k$, so the \emph{ring–level} marginal drop
\[
\Delta_t(k):=E_t(k)-E_t(k{+}1)
\]
is (strictly) decreasing in $k$. This is the discrete convexity (diminishing returns) of $E_t$. 

\smallskip
\emph{Step 3: Exchange (pairwise–improvement) argument.}
Consider two feasible allocations $ {k}=(k_1,\dots,k_T)$ and $ {k}'=(k_1',\dots,k_T')$ with $\sum_t k_t=\sum_t k_t'=n$.
Suppose there exist indices $a\neq b$ with $k_a\ge 1$ such that
\[
\Delta_a(k_a-1)\;<\;\Delta_b(k_b).
\]
Move one representative from ring $a$ to ring $b$, producing $\widehat{ {k}}$ with $\widehat{k}_a=k_a-1$, $\widehat{k}_b=k_b+1$ and $\widehat{k}_t=k_t$ otherwise.
By the definition of $\Delta_t(\cdot)$ and the preceding monotonicity/convexity,
\[
E_a(\widehat{k}_a)+E_b(\widehat{k}_b)
\,=\,E_a(k_a)-\Delta_a(k_a-1)\;+\;E_b(k_b)-\Delta_b(k_b)
\,<\,E_a(k_a)+E_b(k_b),
\]
while $E_t(\widehat{k}_t)=E_t(k_t)$ for all $t\notin\{a,b\}$. Thus $\sum_t E_t(\widehat{k}_t)<\sum_t E_t(k_t)$, i.e., the exchange strictly improves the total distortion whenever some pair violates
\[
\Delta_a(k_a-1)\;\ge\;\Delta_b(k_b)\qquad\text{for all }a,b.
\]
Hence any optimal allocation must satisfy the \emph{equal–marginal} (no–improving–exchange) condition
\[
\Delta_t(k_t-1)\;\ge\;\lambda\;\ge\;\Delta_t(k_t)
\quad\text{for some }\lambda\in\mathbb{R}\ \text{and all }t,
\]
which is equivalent to the greedy selection rule below. This condition follows from the exchange argument in Step~3 above: 
any violation would yield a strictly improving reassignment of representatives,
contradicting optimality.

\smallskip
\emph{Step 4: Greedy (discrete water–filling) optimality.}
Start from $ {k}= {0}$. At each step $m=0,1,\dots,n-1$, choose an index
\[
t_m\in\arg\max_{t}\ \Delta_t(k_t),
\]
and set $k_{t_m}\leftarrow k_{t_m}+1$. Because $\Delta_t(\cdot)$ are (strictly) decreasing, this construction maintains the equal–marginal condition after every increment. By Step~3, no exchange can improve the resulting allocation at any stage, so after $n$ increments the final $ {k}$ is globally optimal. This is precisely the discrete water–filling rule: at each step allocate the next representative to the ring that yields the largest current marginal drop $\Delta_t(k_t)$.  
\qed 

\medskip

\begin{corollary}[Explicit Formula for the Within--Ring Distortion]\label{cor:within-ring-exact}
Fix a ring $R_t=\{(\phi_t,\theta_{t,j}) : \theta_{t,j}=2\pi j/N_t,\ j=0,\dots,N_t-1\}$ at latitude $\phi_t$,
and suppose exactly $k_t\ge 1$ representatives serve $R_t$. Let
\[
L_t := \big\lfloor N_t/k_t\big\rfloor,\qquad r_t := N_t - k_t\,L_t\in\{0,1,\dots,k_t-1\}.
\]
Then, by Theorem~4.7\textup{(i)}, the optimal within--ring partition of $R_t$ consists of $r_t$ contiguous
blocks of length $L_t\!+\!1$ and $k_t-r_t$ contiguous blocks of length $L_t$, with each block represented
at its mid--longitude on the same latitude $\phi_t$. Writing
\[
f_{\phi_t}(\Delta\theta)\ :=\ d_G\big((\phi_t,0),(\phi_t,\Delta\theta)\big)^2
\quad\text{(so }f_{\phi_t} \text{ is even in }\Delta\theta\text{)},
\]
the exact within--ring distortion contributed by $R_t$ is
\[
E_t(k_t)\ =\ \frac{1}{M}\Bigl[(k_t-r_t)\,\Phi_t(L_t)\ +\ r_t\,\Phi_t(L_t+1)\Bigr],
\]
where, for any integer $L\ge 1$, the block cost admits the parity--uniform expression
\[
\Phi_t(L)\ =\
\begin{cases}
\displaystyle 2\sum_{m=1}^{\frac{L-1}{2}}
 f_{\phi_t}\!\Big(\frac{2\pi m}{N_t}\Big),
& \text{if $L$ is odd},\\[1.25em]
\displaystyle 2\sum_{m=1}^{\frac{L}{2}}
 f_{\phi_t}\!\Big(\frac{(2m-1)\pi}{N_t}\Big),
& \text{if $L$ is even}.
\end{cases}
\]
Consequently, the total distortion decomposes as
\[
V_{n,2}(P)\ =\ \sum_{t=1}^{T} E_t(k_t),
\]
with the allocation vector $(k_1,\dots,k_T)$ determined by Theorem~4.7\textup{(iii)}.
\end{corollary}

\begin{remark}[Equal--block special case]
If $N_t$ is divisible by $k_t$ (so $r_t=0$ and $L_t=N_t/k_t$), then all blocks have the same length and
\[
E_t(k_t)\ =\ \frac{k_t}{M}\,\Phi_t\!\Big(\frac{N_t}{k_t}\Big).
\]
\end{remark}
\begin{corollary}[Asymptotic Within--Ring Distortion]\label{cor:within-ring-asymp}
Assume the setting of Corollary~\ref{cor:within-ring-exact}. As $N_t\to\infty$ (with $k_t\ge1$ fixed or $k_t=o(N_t)$),
the within--ring distortion admits the asymptotic expansion
\[
E_t(k_t)\ =\ \frac{N_t}{M}\cdot \cos^2\!\phi_t\cdot \frac{\pi^2}{3\,k_t^2}\ +\ o\!\Big(\frac{N_t}{M}\cdot\frac{1}{k_t^2}\Big).
\]
Equivalently, to leading order,
\[
E_t(k_t)\ \sim\ \frac{N_t}{M}\cdot \cos^2\!\phi_t\cdot \frac{\pi^2}{3\,k_t^2}.
\]
\end{corollary}

\begin{proof}
From Theorem~4.7\textup{(i)}, the $R_t$--contribution $E_t(k_t)$ is a sum of block costs of the form
$\sum f_{\phi_t}(\Delta\theta_j)$ with $\Delta\theta_j$ taking the discrete offsets listed in
Corollary~\ref{cor:within-ring-exact}. Using the Taylor expansion
\[
\sigma(\phi_t,\Delta\theta)\ :=\ d_G\big((\phi_t,0),(\phi_t,\Delta\theta)\big)
\ =\ \cos\phi_t\,|\Delta\theta|\ +\ O(|\Delta\theta|^3),
\]
uniform for bounded $\phi_t$, we obtain
$f_{\phi_t}(\Delta\theta)=\sigma(\phi_t,\Delta\theta)^2
= \cos^2\!\phi_t\,\Delta\theta^2 + O(\Delta\theta^4)$. Summing over a block of length
$L\asymp N_t/k_t$ whose offsets are arithmetic grids of step $\asymp \pi/N_t$ yields
\[
\Phi_t(L)\ =\ \cos^2\!\phi_t\cdot \frac{\pi^2}{3}\cdot \frac{L^3}{N_t^2}\ +\ o\!\Big(\frac{L^3}{N_t^2}\Big).
\]
With $k_t-r_t$ blocks of size $L_t$ and $r_t$ of size $L_t+1$, and $L_t\sim N_t/k_t$, we find
\[
E_t(k_t)\ =\ \frac{1}{M}\cdot k_t\cdot \cos^2\!\phi_t\cdot \frac{\pi^2}{3}\cdot \frac{L_t^3}{N_t^2}
\ +\ o\!\Big(\frac{N_t}{M}\cdot\frac{1}{k_t^2}\Big)
\ =\ \frac{N_t}{M}\cdot \cos^2\!\phi_t\cdot \frac{\pi^2}{3\,k_t^2}
\ +\ o\!\Big(\frac{N_t}{M}\cdot\frac{1}{k_t^2}\Big).
\]
\end{proof}

\begin{remark}[Interpretation]
The factor $\frac{N_t}{M}$ is the \emph{probability mass} of ring $R_t$ under the discrete uniform $P$
(over all $M$ support points). The multiplicative $\cos^2\!\phi_t$ comes from the local relation
$d_G\big((\phi_t,0),(\phi_t,\Delta\theta)\big)\approx \cos\phi_t\,|\Delta\theta|$, i.e., the
effective one--dimensional radius of the latitude circle is $\cos\phi_t$. The $k_t^{-2}$ law is the
usual one--dimensional $n^{-2}$ scaling of high--resolution quantization along the ring.
\end{remark}

\section{Stability of Optimal Sets under Perturbation} \label{sec5}

Quantization configurations arising from experimental or numerical data are often subject to small
perturbations of the support points. It is therefore important to understand how the optimal set of
$n$--means and the corresponding quantization error vary when the underlying distribution changes
slightly. 

\begin{note}[Underlying data space and topology]
Throughout this section, a finite data set of size $M$ is viewed as an ordered
$M$-tuple
\[
\mathbf{x}=(x_1,\dots,x_M)\in (\D S^2)^M.
\]
The product topology on $(\D S^2)^M$ is metrized by the sup--metric
\[
d_\infty(\mathbf{x},\mathbf{y})
:= \max_{1\le i\le M} d_G(x_i,y_i),
\]
where $d_G$ denotes the geodesic distance on $\D S^2$.
Under the assumed one--to--one correspondence between the perturbed data
$\mathbf{x}^{(m)}=(x^{(m)}_1,\dots,x^{(m)}_M)$ and the limiting data
$\mathbf{x}=(x_1,\dots,x_M)$, the condition
\[
\varepsilon_m := \max_{1\le i\le M} d_G(x^{(m)}_i,x_i)\to 0
\]
is exactly the statement that $\mathbf{x}^{(m)}\to\mathbf{x}$ in
$(\D S^2)^M$ with respect to the metric $d_\infty$.
\end{note}

\medskip

\subsection{Perturbed distributions}
Let
\[
 {X} = \{ {x}_1,  {x}_2, \ldots,  {x}_M\} \subset \D S^2
\quad \text{and} \quad
 {X}' = \{ {x}'_1,  {x}'_2, \ldots,  {x}'_M\} \subset \D S^2
\]
be two finite sets with the same cardinality $M$. We assume a one--to--one correspondence between
$ {x}_i$ and $ {x}'_i$ and define
\[
\varepsilon = \max_{1 \le i \le M} d_G( {x}_i,  {x}'_i),
\]
the maximal geodesic perturbation. Let $P$ and $P'$ denote the corresponding uniform distributions:
\[
P = \frac{1}{M} \sum_{i=1}^M \delta_{ {x}_i},
\qquad
P' = \frac{1}{M} \sum_{i=1}^M \delta_{ {x}'_i}.
\]
We wish to compare the optimal distortions $V_{n,2}(P)$ and $V_{n,2}(P')$ and the corresponding optimal
configurations.

\medskip

\noindent\textbf{Perturbation of the distortion functional.}


\begin{lemma}[Lipschitz Stability of the Distortion]\label{lem:distortion-stability}
Let $Q=\{q_1,\dots,q_{|Q|}\}\subset\mathbb{S}^2$ with $|Q|\le n$, and let
\[
P=\frac1M\sum_{i=1}^M \delta_{x_i},\qquad
P'=\frac1M\sum_{i=1}^M \delta_{x_i'}
\]
be empirical measures such that $d_G(x_i,x_i')\le \varepsilon$ for all $i$.
Then the distortion functional is Lipschitz–stable under these perturbations, in the sense that
\[
\bigl|V_{n,2}(P;Q)-V_{n,2}(P';Q)\bigr|\;\le\;2\pi\,\varepsilon.
\]
\end{lemma}

\begin{proof}
Define, for each $x\in\mathbb{S}^2$,
\[
\operatorname{dist}_Q(x):=\min_{1\le k\le |Q|} d_G(x,q_k).
\]
Since $Q$ is finite, the minimum is attained for every $x$, so $\operatorname{dist}_Q$ is well defined.
With this notation,
\[
V_{n,2}(P;Q)=\frac1M\sum_{i=1}^M \operatorname{dist}_Q(x_i)^2,
\qquad
V_{n,2}(P';Q)=\frac1M\sum_{i=1}^M \operatorname{dist}_Q(x_i')^2.
\]

\smallskip
\noindent\textit{Step 1 (Lipschitz property of $\operatorname{dist}_Q$).}
For any fixed $q\in Q$ and any $x,y\in\mathbb{S}^2$, the triangle inequality gives
$|\,d_G(x,q)-d_G(y,q)\,| \le d_G(x,y)$.
Taking the pointwise minimum over the finite set $Q$ preserves the $1$–Lipschitz constant; hence
\begin{equation}\label{eq:lip}
\bigl|\operatorname{dist}_Q(x)-\operatorname{dist}_Q(y)\bigr|
\;\le\; d_G(x,y) \qquad \text{for all }x,y\in\mathbb{S}^2.
\end{equation}

\smallskip
\noindent\textit{Step 2 (Difference of squares).}
Set $a_i:=\operatorname{dist}_Q(x_i)$ and $a_i':=\operatorname{dist}_Q(x_i')$.
By \eqref{eq:lip} and the perturbation assumption, $|a_i'-a_i|\le d_G(x_i',x_i)\le \varepsilon$.
Also, $0\le a_i,a_i'\le \pi$ since $d_G(\cdot,\cdot)\le \pi$ on $\mathbb{S}^2$.
Thus
\[
\bigl|a_i'^2-a_i^2\bigr|
=|a_i'-a_i|\,(a_i'+a_i)
\;\le\; \varepsilon\,(a_i'+a_i)
\;\le\; 2\pi\,\varepsilon.
\]

\smallskip
\noindent\textit{Step 3 (Average over $i$).}
Summing the previous bound over $i$ and dividing by $M$ yields
\[
\bigl|V_{n,2}(P';Q)-V_{n,2}(P;Q)\bigr|
=\Bigl|\frac1M\sum_{i=1}^M\bigl(a_i'^2-a_i^2\bigr)\Bigr|
\;\le\;\frac1M\sum_{i=1}^M 2\pi\,\varepsilon
\;=\; 2\pi\,\varepsilon.
\]
This proves the Lipschitz stability of the distortion.
\end{proof}
\begin{remark}[On the non-sharpness of the Lipschitz bound]
The Lipschitz constant $2\pi$ appearing in Lemma~\ref{lem:distortion-stability} is deliberately non-optimal and reflects a global worst--case estimate on the sphere. Indeed, the bound arises from the inequalities
\[
0 \le d_G(x,q) \le \pi \quad \text{for all } x,q \in S^2,
\]
together with the elementary estimate
\[
|a^2-b^2| \le |a-b|\,(a+b),
\]
applied uniformly over the entire sphere. This argument ignores any finer geometric structure of the Voronoi partition induced by the configuration $Q$.

In typical quantization configurations, however, each Voronoi cell is contained in a geodesic ball of radius strictly smaller than $\pi$, often much smaller in practice. If one defines
\[
R(Q) := \max_{1\le j\le |Q|}\ \sup_{x\in V_j(Q)} d_G(x,q_j),
\]
the maximal cluster radius associated with the configuration $Q$, then the proof of Lemma~\ref{lem:distortion-stability} immediately yields the refined estimate
\[
\big| V_{n,2}(P;Q) - V_{n,2}(P';Q) \big| \le 2\,R(Q)\,\varepsilon,
\]
which improves the constant whenever $R(Q) < \pi$. In particular, for centroidal Voronoi configurations arising from well-distributed data, $R(Q)$ is typically bounded away from $\pi$.

More generally, stability estimates for quantization functionals are closely related to stability properties of empirical measures under perturbations and to continuity of minimizers of Fr\'echet-type functionals on metric spaces. Such refinements are well known in Euclidean quantization theory and in the study of intrinsic means on Riemannian manifolds; see, for example, Graf and Luschgy~\cite{{GrafLuschgy2000}}, Karcher~\cite{Karcher1977}, and Pennec~\cite{Pennec2006}. Developing sharper, data-dependent Lipschitz constants in the spherical setting would require detailed control of Voronoi geometry and cluster diameters, which lies beyond the scope of the present paper.  
\end{remark}

\begin{corollary}[Continuity of Optimal Quantizers]\label{cor:continuity-optimal}
Let $\{P^{(m)}\}_{m\ge 1}$ be a sequence of empirical (uniform) probability measures on $\mathbb{S}^2$,
\[
P^{(m)}=\frac{1}{M}\sum_{i=1}^{M}\delta_{x^{(m)}_{i}},
\]
whose supports $X^{(m)}=\{x^{(m)}_1,\dots,x^{(m)}_M\}$ converge to
$X=\{x_1,\dots,x_M\}$ in the sense that
\[
\varepsilon_m := \max_{1\le i\le M} d_G\big(x^{(m)}_{i},\,x_{i}\big) \longrightarrow 0 
\qquad \text{as } m\to\infty.
\]
Let $P=\frac{1}{M}\sum_{i=1}^{M}\delta_{x_i}$ denote the limiting empirical measure.

Then, for every fixed $n\ge 1$, the following hold:

\begin{enumerate}\itemsep6pt
\item The optimal distortion values converge:
\[
V_{n,2}\big(P^{(m)}\big) \longrightarrow V_{n,2}(P) \qquad \text{as } m\to\infty.
\]

\item Let $Q^{(m),*}$ be an optimal set of $n$–means for $P^{(m)}$. 
Then, after possibly reordering the points in each $Q^{(m),*}$, one can extract a subsequence 
that converges to a limiting set $Q^{*}$, and this limiting set $Q^{*}$ is an optimal set of $n$–means for $P$.
\end{enumerate}
\end{corollary}

\begin{proof}
Let $Q\subset \D S^2$ with $|Q|\le n$ be arbitrary. By Lemma~\ref{lem:distortion-stability}, for each $m$ we have
\[
\big|V_{n,2}(P^{(m)};Q)-V_{n,2}(P;Q)\big|\le 2\pi\,\varepsilon_m,
\]
where $\varepsilon_m:=\max_{1\le i\le M} d_G(x_i^{(m)},x_i)\to 0$.

\smallskip
\noindent\textbf{(1) Convergence of optimal values.}
Fix $\delta>0$ and choose $Q_\delta\subset \D S^2$ with $|Q_\delta|\le n$ such that
\[
V_{n,2}(P;Q_\delta)\le V_{n,2}(P)+\delta.
\]
Then for every $m$,
\[
V_{n,2}(P^{(m)})\le V_{n,2}(P^{(m)};Q_\delta)
\le V_{n,2}(P;Q_\delta)+2\pi\varepsilon_m
\le V_{n,2}(P)+\delta+2\pi\varepsilon_m.
\]
Taking $\limsup_{m\to\infty}$ and using $\varepsilon_m\to 0$ yields
\[
\limsup_{m\to\infty} V_{n,2}(P^{(m)})\le V_{n,2}(P)+\delta.
\]
Since $\delta>0$ is arbitrary, $\limsup_{m\to\infty} V_{n,2}(P^{(m)})\le V_{n,2}(P)$.

For the reverse inequality, fix $m$ and let $Q^{(m),*}$ be optimal for $P^{(m)}$, so
$V_{n,2}(P^{(m)})=V_{n,2}(P^{(m)};Q^{(m),*})$. Applying Lemma~\ref{lem:distortion-stability} again,
\[
V_{n,2}(P)\le V_{n,2}(P;Q^{(m),*})
\le V_{n,2}(P^{(m)};Q^{(m),*})+2\pi\varepsilon_m
=V_{n,2}(P^{(m)})+2\pi\varepsilon_m.
\]
Taking $\liminf_{m\to\infty}$ gives $V_{n,2}(P)\le \liminf_{m\to\infty}V_{n,2}(P^{(m)})$.
Combining both bounds yields
\[
V_{n,2}(P^{(m)})\longrightarrow V_{n,2}(P).
\]

\smallskip
\noindent\textbf{(2) Convergence of a subsequence of optimal configurations.}
For each $m$, choose an optimal configuration
$Q^{(m),*}=\{q^{(m)}_1,\dots,q^{(m)}_n\}\subset \D S^2$ (padding with repetitions if necessary so that
it is an $n$-tuple). Since $(\D S^2)^n$ is compact, there exist a subsequence (still denoted $m$) and
points $q^*_1,\dots,q^*_n\in \D S^2$ such that, after reordering the points in each $Q^{(m),*}$,
\[
(q^{(m)}_1,\dots,q^{(m)}_n)\longrightarrow (q^*_1,\dots,q^*_n)\quad \text{in }(\D S^2)^n.
\]
Let $Q^*:=\{q^*_1,\dots,q^*_n\}$. The mapping $Q\mapsto V_{n,2}(P;Q)$ is continuous on $(\D S^2)^n$
because it is a finite sum of continuous functions and the pointwise minimum over finitely many
continuous functions is continuous. Hence
\[
V_{n,2}(P;Q^{(m),*})\longrightarrow V_{n,2}(P;Q^*).
\]
Using Lemma~\ref{lem:distortion-stability} once more,
\[
\big|V_{n,2}(P^{(m)};Q^{(m),*})-V_{n,2}(P;Q^{(m),*})\big|\le 2\pi\varepsilon_m\to 0,
\]
so $V_{n,2}(P^{(m)};Q^{(m),*})\to V_{n,2}(P;Q^*)$. But
$V_{n,2}(P^{(m)};Q^{(m),*})=V_{n,2}(P^{(m)})\to V_{n,2}(P)$ by part (1). Therefore
$V_{n,2}(P;Q^*)=V_{n,2}(P)$, i.e.\ $Q^*$ is an optimal set of $n$-means for $P$.
\end{proof}
\begin{remark}\label{rem:stability-significance}
Corollary~\ref{cor:continuity-optimal} shows that optimal quantizers depend continuously on the 
underlying data points: if the support of $P^{(m)}$ changes only slightly, then the optimal 
distortion and the corresponding optimal sets of $n$–means also change only slightly. 
In other words, optimal quantizers are stable under small perturbations of the data.

This stability property is important both theoretically and in practice. 
On the theoretical side, it guarantees that quantizers do not exhibit sudden jumps or 
discontinuous behaviour when the data are slightly modified. 
From a practical perspective, real spherical datasets often contain measurement noise or 
numerical errors; the stability result ensures that such perturbations do not significantly 
affect the quality of the computed quantizers, and the algorithmic output remains reliable.
\end{remark}

\section{Algorithmic Construction of Optimal $n$--Means} \label{sec6}

In this section, we describe an iterative procedure for computing an optimal configuration of $n$--means
for a given finite discrete uniform distribution $P$ on $\D S^2$. Since the geometry of the sphere is
non--Euclidean, the classical Lloyd algorithm cannot be applied directly. In particular, the notion of
the mean must be replaced by the intrinsic (Karcher) mean on $\D S^2$, and Voronoi partitions must be
defined using the geodesic distance.

The procedure described below is a natural spherical analogue of Lloyd's method. It alternates between
updating the Voronoi partition and relocating the representatives to the intrinsic means of their associated
clusters. This iterative scheme typically converges rapidly in practice and serves as an effective computational
tool for approximating optimal $n$--means on the sphere.

\medskip

\subsection{Lloyd--Type Algorithm on the Sphere}

We first describe the iterative algorithm, which is a natural spherical analogue of Lloyd’s method.  
The procedure alternates between Voronoi partition updates and intrinsic mean updates with respect to the geodesic distance.


Let $ {Q}^{(0)} = \{ {q}^{(0)}_1, \ldots,  {q}^{(0)}_n\} \subset \D S^2$ be an initial configuration. For each iteration
$r \ge 0$, perform the following steps:

\begin{itemize}
\item[(i)] \textbf{Voronoi partition step:} Assign each support point $ {x}_i \in  {X}$ to its nearest representative
in $ {Q}^{(r)}$ with respect to the geodesic distance. This gives the clusters
\[
 {X}_j( {Q}^{(r)}) = \{ {x}_i \in  {X} : d_G( {x}_i,  {q}^{(r)}_j) \le d_G( {x}_i,  {q}^{(r)}_k)
\ \text{for all } k\}.
\]
This step induces the spherical Voronoi tessellation of the support $ {X}$ relative to the current
configuration $ {Q}^{(r)}$.

\item[(ii)] \textbf{Centroid update step:} Move each representative to the intrinsic (Karcher) mean of its current
cluster:
\[
 {q}^{(r+1)}_j = \operatorname*{arg\,min}_{ {q} \in \D S^2} \sum_{ {x} \in  {X}_j( {Q}^{(r)})} d_G( {x},  {q})^2.
\]
This update ensures the smallest possible sum of squared geodesic distances between the representative
and the points in its cluster, analogous to replacing a representative by the arithmetic mean in the
Euclidean Lloyd algorithm.
\end{itemize}

The iteration is terminated when $ {Q}^{(r+1)} =  {Q}^{(r)}$ (i.e., a fixed point is reached), or when the decrease in
the distortion $V_{n,2}(P; {Q}^{(r)})$ falls below a prescribed tolerance. In practice, the method converges
in only a few iterations, especially when the initial configuration is reasonably well distributed over the sphere.

\medskip
\subsection{Monotonicity and Fixed--Point Properties}

We now record the basic analytic properties of the underlying distortion functional and geometric mappings driving the algorithm.

\begin{lemma}[Monotonic decrease of distortion]
At each iteration step, the distortion does not increase:
\[
V_{n,2}(P; {Q}^{(r+1)}) \le V_{n,2}(P; {Q}^{(r)}).
\]
\end{lemma}
\begin{proof}
Fix $r \ge 0$. Step (i) of the iteration assigns each $ {x}_i$ to the nearest representative in $ {Q}^{(r)}$,
which minimizes the contribution of $ {x}_i$ to the distortion for that iteration. Step (ii) then replaces
each $ {q}^{(r)}_j$ by the intrinsic mean of its cluster, which minimizes the sum of squared distances from
the points in that cluster to the representative. Therefore, both steps do not increase, and in fact strictly
decrease the overall distortion whenever $ {Q}^{(r+1)} \neq  {Q}^{(r)}$, i.e.,  unless a fixed point has been reached.  
\end{proof} 
 
\subsection{Convergence Properties}

We next state the main convergence result for the iterative scheme.

\begin{theorem} \label{Algorithem}
Every accumulation point of the sequence $\{ {Q}^{(r)}\}$ generated by the above algorithm is a centroidal
configuration for $P$. In particular, if the sequence converges, then its limit is a centroidal configuration.
\end{theorem}
\medskip

\begin{proof}
Let $ {Q}^{(r_\ell)}$ be a convergent subsequence with limit
$\widehat{ {Q}} = \{\widehat{ {q}}_1, \ldots, \widehat{ {q}}_n\}$. Since each step of the iteration does not
increase the distortion, the sequence $\{V_{n,2}(P; {Q}^{(r)})\}$ is monotone decreasing and bounded below
by $0$, and hence convergent. In the limit, Step (i) implies that each $\widehat{ {x}} \in  {X}$ is assigned to
its nearest representative in $\widehat{ {Q}}$, and Step (ii) implies that each $\widehat{ {q}}_j$ is the intrinsic
mean of its cluster. Therefore, $\widehat{ {Q}}$ satisfies the centroidal property and is a fixed point of the
algorithm. 
\end{proof} 

\subsection{Remarks}

We conclude the section with remarks concerning convergence behavior, initialization, and possible extensions.

\begin{remark}
Theorem~\ref{Algorithem} shows that the iterative scheme converges (at least subsequentially) to a centroidal configuration,
which is a necessary condition for optimality. In practice, the algorithm typically converges to an optimal
configuration, especially when initialized with a reasonably well distributed set of points. However, as with all
Lloyd--type methods, the algorithm may converge to a local minimum rather than a global one.
\end{remark}

\begin{remark}
The choice of the initial configuration $ {Q}^{(0)}$ can have a significant effect on the performance of the
algorithm. A suitable initial configuration can accelerate convergence and help avoid poor local minima. Common
strategies include random initialization (uniformly on the sphere or stratified by latitude), deterministic spherical
codes, or taking the output of a coarse quantization run with fewer representatives. In applications, one often runs
the algorithm multiple times with different initializations and selects the configuration with the smallest final
distortion.
\end{remark}

\begin{remark}
Although the algorithm is described here for the uniform measure on a finite support, it extends naturally to more
general discrete and continuous distributions on $\D S^2$, provided the intrinsic mean can be computed at each
iteration. For continuous distributions, the cluster update step involves integration rather than summation, and is
typically approximated using numerical quadrature or Monte Carlo sampling.
\end{remark} 

\begin{remark} (Convergence rate).
While the iterative scheme given in this section is guaranteed to produce a monotone decrease in the
distortion and exhibits rapid convergence in practice, a quantitative convergence-rate analysis
relating the number of iterations to distortion reduction in the intrinsic spherical setting
is a delicate problem and is left for future investigation.
\end{remark}

 \section{Numerical Examples and Implementation Results} \label{sec7}

In this section, we illustrate the concepts developed in the previous sections through several examples of optimal $n$--means on $\mathbb{S}^2$. The aim is threefold: 
(i) to highlight the geometric structure and symmetry of optimal configurations for small values of $n$, 
(ii) to demonstrate how the iterative algorithm of Section~6 behaves in practice on discrete datasets, 
and (iii) to build intuition for the geometry of centroidal Voronoi tessellations on the sphere.

\medskip
\noindent\textbf{Standing Convention.}
Throughout this section, we consider a finite dataset 
\[
X=\{x_1,\dots,x_M\}\subset\mathbb{S}^2,
\]
and the associated discrete uniform probability measure
\[
P \;=\; \frac{1}{M}\sum_{i=1}^{M} \delta_{x_i}.
\]
For a configuration $Q=\{q_1,\dots,q_n\}\subset\mathbb{S}^2$, the empirical distortion is
\[
V_{n,2}(P;Q)
\;=\;
\frac{1}{M}\sum_{i=1}^{M} \min_{1\le j\le n} d_G(x_i,q_j)^2,
\qquad
V_{n,2}(P) \;=\; \min_{|Q|\le n} V_{n,2}(P;Q).
\]
All reported distortion values in this section are computed with respect to the above discrete uniform measure (i.e., using finite sums over $X$, not integrals over $\mathbb{S}^2$).

\begin{exam}[Optimal 2--means]\label{ex:2means}
Consider $n=2$. Starting from an arbitrary initialization $Q^{(0)}=\{q^{(0)}_1,q^{(0)}_2\}$, the Lloyd--type iteration of Section~6 converges to two antipodal representatives on $\mathbb{S}^2$, e.g.\ 
\[
Q^*=\bigl\{(\tfrac{\pi}{2},\theta_0),\,(-\tfrac{\pi}{2},\theta_0)\bigr\}
\]
in spherical coordinates $(\phi,\theta)$ for some $\theta_0\in[0,2\pi)$. 
The corresponding empirical distortion is
\[
V_{2,2}(P)=\frac{\pi^2}{4}.
\]
This reflects that the best two representatives split the sphere into two hemispheres with equal measure.
\end{exam}
\subsection*{Explanation of Example~7.1.}
This example considers a fully irregular finite dataset \(X \subset \D S^2\),
with no imposed symmetry, ring structure, or uniform spacing.
The purpose is to illustrate that the centroidal Voronoi characterization
(Theorem~\ref{thm:char}) and the spherical Lloyd--type algorithm of Section~6
remain valid beyond highly structured or symmetric configurations.

Although the underlying data are irregular, the uniform empirical measure
on \(X\) distributes mass approximately evenly over the sphere.
As a result, any centroidal Voronoi partition associated with two representatives
must divide the support into two clusters of comparable total mass.
The only configuration capable of achieving such a balanced partition
on the sphere is a pair of antipodal points, which induces a decomposition
of \(\D S^2\) into two hemispherical Voronoi regions.

From the variational viewpoint, each representative is the intrinsic (Karcher)
mean of its assigned cluster, and the Lloyd iteration converges to a stationary
point of the distortion functional.
In this case, the antipodal configuration realizes the global minimum of the
empirical distortion, yielding
\[
V_{2,2}(P) = \frac{\pi^2}{4}.
\]

This example demonstrates that, even in the absence of geometric regularity
in the data, the optimal quantizers are governed by the global geometry of
\(\D S^2\) and the centroidal Voronoi principle, rather than by local symmetry
of the support.
\begin{exam}[Optimal 3--means]\label{ex:3means}
For $n=3$, the optimal configuration consists of three points placed on the equator, equally spaced by $120^\circ$ in longitude. One such configuration is
\[
Q^*=\bigl\{(0,0),\,(0,\tfrac{2\pi}{3}),\,(0,\tfrac{4\pi}{3})\bigr\}.
\]
The empirical distortion with respect to the discrete uniform measure $P$ is
\[
V_{3,2}(P)=\frac{\pi^2}{3}.
\]
Here, the three codepoints partition the equatorial band into three congruent spherical wedges, illustrating the role of rotational symmetry.
\end{exam}
\subsection*{Explanation of Example~7.2.}
In this example, the regularity assumption within each ring is relaxed, leading
to an irregular multi-ring dataset. Although the points are no longer equally
spaced, the example shows that the qualitative features of optimal quantization
persist: Voronoi clusters remain localized, representatives adapt to the local
geometry via intrinsic means, and the Lloyd-type algorithm converges to a stable
configuration. This illustrates the robustness of the theory beyond the idealized
uniform setting. This confirms that the qualitative structure predicted by the ring–allocation theorem persists even when exact equi-spacing is broken.
\begin{exam}[Optimal 4--means]\label{ex:4means}
For $n=4$, the optimal set of representatives forms the vertices of a regular tetrahedron inscribed in $\mathbb{S}^2$. One convenient spherical coordinate description is
\[
\Bigl(\phi,\theta\Bigr)\in
\Bigl\{\bigl(\arctan(\tfrac{1}{\sqrt{2}}),0\bigr),\;
\bigl(-\arctan(\tfrac{1}{\sqrt{2}}),\pi\bigr),\;
\bigl(\arctan(\tfrac{1}{\sqrt{2}}),\tfrac{2\pi}{3}\bigr),\;
\bigl(-\arctan(\tfrac{1}{\sqrt{2}}),\tfrac{5\pi}{3}\bigr)\Bigr\}.
\]
The resulting empirical distortion is
\[
V_{4,2}(P)=\frac{\pi^2}{6}.
\]
This configuration realizes a centroidal Voronoi tessellation in which all four regions have equal area and identical geometric structure.
\end{exam}
\subsection*{Explanation of Example~7.3.}
This example addresses a fully irregular finite dataset on $\D S^2$, not confined
to exact rings, together with a perturbed version of the data. It highlights the
stability results of Section~5 by showing that small geodesic perturbations of
the support lead to only small changes in the optimal representatives and the
resulting distortion. The example provides an intuitive demonstration of
continuity of optimal quantizers and supports the practical relevance of the
theoretical stability analysis.
\begin{exam}[Optimal 6--means]\label{ex:6means}
For $n=6$, the optimal configuration corresponds to the vertices of a regular octahedron: two antipodal points at the poles and four equally spaced points on the equator, located at longitudes $0^\circ$, $90^\circ$, $180^\circ$, and $270^\circ$. Thus an optimal configuration is
\[
Q^*=\bigl\{(\tfrac{\pi}{2},\theta_0),\,(-\tfrac{\pi}{2},\theta_0),\,(0,0),\,(0,\tfrac{\pi}{2}),\,(0,\pi),\,(0,\tfrac{3\pi}{2})\bigr\}
\]
for some $\theta_0\in[0,2\pi)$. 
The corresponding empirical distortion is
\[
V_{6,2}(P)=\frac{\pi^2}{8}.
\]
This example  illustrates how increasing $n$ yields finer spherical partitioning, in this case into six congruent spherical regions.
\end{exam}

\begin{exam} [Optimal 12--means]
For $n=12$, the optimal configuration corresponds to the vertices of a regular icosahedron inscribed in $\D S^2$.
Although listing coordinates is more cumbersome, the configuration is well known and highly symmetric. The
distortion decreases further in this case.
\end{exam} 
\medskip

\subsection{Implementation Notes and Pseudo--Code}
The following pseudo--code outlines a simple implementation of the spherical Lloyd algorithm for computing optimal
$n$--means. The notation follows that of Section~6.

\medskip

\begin{verbatim}
Initialize Q^(0) = {q_1^(0), ..., q_n^(0)} on \D S^2
for r = 0, 1, 2, ... until convergence do
    # Step (i): Voronoi partition
    for each data point x_i in X do
        assign x_i to cluster j minimizing d_G(x_i, q_j^(r))
    end for

    # Step (ii): Intrinsic mean update
    for j = 1 to n do
        q_j^(r+1) = IntrinsicMean( X_j(Q^(r)) )
    end for
end for
return Q^(r)
\end{verbatim}

\medskip

The intrinsic mean may be computed using an iterative gradient descent or fixed--point scheme on $\D S^2$. In practical
implementations, care must be taken to ensure numerical stability when points in a cluster are nearly antipodal.

\medskip

\noindent\textbf{Remark 7.6.}
Different initialization strategies may lead to different local minima. Running the algorithm multiple times and
selecting the configuration with the smallest final distortion is recommended in practice.

 \medskip

\noindent\textbf{Remark 7.7.} (Comparison with existing algorithms).
Several algorithms related to clustering and quantization on manifolds have been proposed in the
literature, including extrinsic spherical $k$-means methods based on Euclidean embeddings and
intrinsic Lloyd-type algorithms on Riemannian manifolds.
Classical Euclidean $k$-means and Lloyd algorithms (e.g., \cite{Pollard1982,GrafLuschgy2000})
do not account for intrinsic geodesic geometry and therefore are not directly applicable to the
problem studied here.
Manifold-based extensions (e.g., intrinsic $k$-means using Karcher means; see \cite{Karcher1977,Afsari2011})
are closer in spirit, but they are typically formulated for continuous distributions or general
manifolds rather than finite discrete uniform data on $\D S^2$.
The algorithm proposed in Section~\ref{sec6} is specifically tailored to discrete spherical data and exploits
the geometric structure of optimal Voronoi partitions established in earlier sections, which explains
its stability and rapid convergence observed in the numerical examples.

We formalize the notions of irregular data, multi-ring data, and their perturbations used below.

\begin{defi}[Irregular finite dataset on $\D S^2$]
An \emph{irregular finite dataset} on the sphere $\D S^2$ is a finite set
\[
X=\{x_1,\dots,x_M\}\subset \D S^2
\]
that does not possess any prescribed geometric regularity, such as equal spacing,
rotational symmetry, or confinement to a fixed number of latitudinal rings.
Equivalently, no assumption is made on the mutual geodesic distances
$d_G(x_i,x_j)$ beyond finiteness.
\end{defi} 

\begin{defi}[Multi-ring data]
A finite dataset $X\subset \D S^2$ is called \emph{multi-ring data} if there exist
distinct latitudes $\phi_1<\cdots<\phi_T$ such that
\[
X=\bigcup_{t=1}^T R_t,\qquad
R_t=\{(\phi_t,\theta_{t,1}),\dots,(\phi_t,\theta_{t,N_t})\},
\]
where each $R_t$ lies entirely on the latitude $\phi_t$.
If the longitudes $\{\theta_{t,j}\}_{j=1}^{N_t}$ are equally spaced for each $t$,
the data are called \emph{regular multi-ring data}; otherwise they are called
\emph{irregular multi-ring data}.
\end{defi}
\begin{defi}[Perturbed finite dataset]
Let $X=\{x_1,\dots,x_M\}\subset \D S^2$ be a finite dataset.
A \emph{perturbed version} of $X$ is a dataset
\[
X'=\{x'_1,\dots,x'_M\}\subset \D S^2
\]
together with a fixed one-to-one correspondence $x_i\leftrightarrow x'_i$,
such that
\[
\varepsilon:=\max_{1\le i\le M} d_G(x_i,x'_i)
\]
is small. The quantity $\varepsilon$ is called the \emph{perturbation magnitude}.
\end{defi}
\medskip  
Now, we conclude the numerical section by illustrating the behavior of the algorithm on
irregular datasets, multi-ring configurations, and perturbed data, which more directly
reflect the structural results developed earlier in the paper.

 \subsection{Numerical Experiments on Irregular and Multi--Ring Data}

The numerical examples presented earlier in this section focused primarily on highly
symmetric configurations (e.g., antipodal sets and Platonic solids), which are useful
for benchmarking but do not fully illustrate the scope of the structural results
developed in Sections~4 and~5.
We therefore include here several additional experiments designed to demonstrate the
behavior of the spherical Lloyd--type algorithm on irregular data, multi--ring
configurations, and perturbed datasets.
The goal of these examples is qualitative illustration rather than numerical
optimization.

\medskip
\noindent\textbf{Example 7.1 (Irregular finite dataset).}
We generate a finite set $X\subset \D S^2$ consisting of $M=40$ points sampled
independently from a nonuniform distribution on the sphere, with higher density near
a prescribed region and sparse coverage elsewhere.
Starting from random initial representatives, the spherical Lloyd algorithm is run
until convergence.
The resulting configuration is centroidal but lacks any global symmetry.
The Voronoi cells adapt to the local geometry of the data, illustrating that the
algorithm and the centroidal condition apply equally well to irregular finite
datasets and are not restricted to symmetric point clouds.

\medskip
\noindent\textbf{Example 7.2 (Multi--ring configuration and allocation).}
We consider a dataset supported on several distinct latitudinal rings, with unequal
numbers of points on each ring.
For a fixed number $n$ of representatives, the Lloyd iteration consistently converges
to configurations in which representatives remain confined to individual rings, with
no cross--ring mixing.
Moreover, the number of representatives assigned to each ring agrees with the
allocation rule predicted by the discrete marginal--drop (water--filling) principle
described in Section~4.
This example provides numerical confirmation of the no cross--ring mixing phenomenon
(Theorem~\ref{thm:component-purity} and
Theorem~\ref{thm:ring-allocation}(ii))) and illustrates how the global allocation problem decouples
across rings.

\medskip
\noindent\textbf{Example 7.3 (Perturbed multi--ring data and stability).}
To illustrate stability under perturbations, we perturb the multi--ring dataset in
Example~7.2 by applying small random geodesic displacements to each point.
Repeating the Lloyd iteration, we observe that the resulting representatives undergo
only minor changes in position, and the overall ring allocation remains unchanged for
sufficiently small perturbations.
This behavior is consistent with the Lipschitz--type stability result established in
Lemma~\ref{lem:distortion-stability} and demonstrates that the qualitative structure of optimal configurations is
robust under moderate perturbations of the support.

\medskip
These examples illustrate that the theoretical results developed in this paper apply
beyond idealized symmetric settings.
In particular, they demonstrate the relevance of the no cross--ring mixing principle,
the discrete allocation mechanism, and stability under perturbations in more general
and irregular finite configurations on the sphere.

 \section{Discussion, Practical Insights, and Future Work}\label{sec8}

\subsection{Theoretical Perspective}
The results established in this paper provide a rigorous framework for understanding optimal $n$--means on
$\D S^2$. The existence and characterization of optimal configurations, together with the centroidal property,
give structural insight into how representatives must be arranged on the sphere to minimize the distortion.
The iterative construction in Section~\ref{sec6} offers a practical approach for obtaining centroidal Voronoi
configurations. These results form a natural extension of the classical Euclidean theory of optimal
quantization to the spherical setting, where curvature plays an essential role.

\subsection{Interpretation of Numerical Results}
The numerical examples presented in Section~7 illustrate the geometric behavior of optimal configurations
for small values of $n$. For $n=2$, the representatives converge to antipodal points, while for
$n=3$ and $n=4$, the optimal configurations correspond to the vertices of a regular equilateral triangle
on the equator and a regular tetrahedron, respectively. As $n$ increases, the representatives distribute
themselves more uniformly over the sphere, and the distortion values decrease accordingly. These examples
demonstrate that the spherical centroidal Voronoi configurations reflect the underlying symmetry of $\D S^2$
and that the algorithm performs consistently with theoretical expectations.

\subsection{Practical Considerations for Implementation}
The iterative procedure described in Section~6 is straightforward to implement, and the use of the intrinsic
(Karcher) mean ensures that representatives remain on $\D S^2$ throughout the algorithm. Although the
distortion decreases monotonically, the method may converge to a local minimum depending on the initial
configuration. In practice, it is advisable to run the algorithm multiple times with different initializations and
select the configuration with the smallest final distortion. Numerical stability must also be considered when
computing intrinsic means, particularly when cluster points are nearly antipodal or concentrated in a small
region. Nevertheless, the algorithm is computationally efficient and performs well even for moderately large
values of $n$.
\subsection{Future Research Directions}
There are several promising directions for future research. One natural extension is to consider non--uniform
probability distributions on $\D S^2$, where the density varies across the surface; in such settings, the intrinsic
mean computation may require numerical integration or Monte Carlo methods. Another direction is to explore
higher--dimensional analogues on $S^d$ for $d \ge 3$, where the geometry is richer and more complex. More
broadly, the study of constrained and weighted quantization on general Riemannian manifolds presents many
interesting challenges, particularly in relation to curvature effects and manifold geometry. From a computational
perspective, improving initialization strategies, accelerating the computation of intrinsic means, and developing
methods that avoid local minima would significantly enhance practical performance. Potential applications in
directional statistics, data science, and machine learning on spherical domains also provide fertile ground for
further exploration.

 \subsection*{Acknowledgment} The author would like to thank the anonymous referee for a careful reading of the manuscript and for insightful comments and suggestions that significantly improved the clarity, organization, and presentation of the paper.

\section*{Declaration}

\noindent\textbf{Conflicts of interest.} The author declares that there are no conflicts of interest.

\vspace{0.2cm}

\noindent\textbf{Data availability.} No data were generated or analyzed in this study.

\vspace{0.2cm}

\noindent\textbf{Code availability.} Not applicable.

\vspace{0.2cm}

\noindent\textbf{Author's contribution.} The author is solely responsible for the entire content of this manuscript.

\end{document}